\documentclass[11pt]{article}
\usepackage[numbers,square]{natbib}
\usepackage{amsmath}
\usepackage{amsthm}
\usepackage{amsfonts}
\usepackage{amssymb}
\usepackage{siunitx}
\usepackage{commath}
\usepackage{graphicx}
\usepackage{xspace}
\usepackage{color}
\usepackage[lofdepth,lotdepth]{subfig}
\usepackage{stmaryrd}
\usepackage{url}
\usepackage{amsthm}
\usepackage{footnote}
\makesavenoteenv{tabular}
\makesavenoteenv{table}
\usepackage[hidelinks]{hyperref}
\hypersetup{colorlinks = true, urlcolor = blue,
  linkcolor = blue, citecolor = blue}
\usepackage{cleveref}

\usepackage[margin=0.7in]{geometry}
\makeatletter
\newcommand{\tnorm}{\@ifstar\@tnorms\@tnorm}
\newcommand{\@tnorms}[1]{%
  \left|\mkern-1.5mu\left|\mkern-1.5mu\left|
   #1
  \right|\mkern-1.5mu\right|\mkern-1.5mu\right|
}
\newcommand{\@tnorm}[2][]{%
  \mathopen{#1|\mkern-1.5mu#1|\mkern-1.5mu#1|}
  #2
  \mathclose{#1|\mkern-1.5mu#1|\mkern-1.5mu#1|}
}
\makeatother

\newcommand{\jump}[1]{\llbracket #1 \rrbracket}
\newcommand{\av}[1]{\{\!\!\{#1\}\!\!\}}
\newtheorem{theorem}{Theorem}
\newtheorem{lemma}{Lemma}

\newtheorem{remark}{Remark}
\title{A hybridizable discontinuous Galerkin method for the coupled
  Navier--Stokes and Darcy problem}
\author{A. Cesmelioglu\thanks{Department of Mathematics and
    Statistics, Oakland University, MI, USA
    (\url{cesmelio@oakland.edu}),
    \url{https://orcid.org/0000-0001-8057-6349}} \and
  S. Rhebergen\thanks{Department of Applied Mathematics, University of
    Waterloo, ON, Canada (\url{srheberg@uwaterloo.ca}),
    \url{http://orcid.org/0000-0001-6036-0356}}}
\begin{document}
\maketitle
\begin{abstract}
  We present and analyze a strongly conservative hybridizable
  discontinuous Galerkin finite element method for the coupled
  incompressible Navier--Stokes and Darcy problem with
  Beavers--Joseph--Saffman interface condition. An \emph{a priori}
  error analysis shows that the velocity error does not depend on the
  pressure, and that velocity and pressure converge with optimal
  rates. These results are confirmed by numerical examples.
\end{abstract}
\section{Introduction}
\label{sec:introduction}

We consider the solution of the coupled incompressible Navier--Stokes
and Darcy problem with Beavers--Joseph--Saffman interface
condition. The Navier--Stokes equations describe the motion of a
freely flowing incompressible fluid in one sub-region of the
domain. These equations are coupled by the Beavers--Joseph--Saffman
interface condition to the Darcy equations that describe the flow of a
fluid in porous media, the second sub-region of our domain. For the
analysis and applications of these equations we refer to
\cite{Discacciati:2009}.

A conforming finite element method for the coupled
Navier--Stokes/Darcy problem was proposed and analyzed by Badea et
al. \cite{Badea:2010}, where they consider a Taylor--Hood
discretization of the Navier--Stokes equations and use quadratic
Lagrangian elements for the primal formulation of the Darcy
equation. Girault and Rivi\`ere \cite{Girault:2009} consider a primal
form of the Darcy equations coupled to the velocity-pressure
formulation of the Navier--Stokes equations. They propose a
discontinuous Galerkin (DG) method based on an upwind Lesaint--Raviart
DG discretization of the convective terms in the Navier--Stokes
equations and non-symmetric, symmetric, and incomplete interior
penalty Galerkin methods for the diffusion terms in the Navier--Stokes
equations and for the primal form of the Darcy equation. Later, using
the dual-mixed formulation of the Darcy equation, a conforming mixed
finite element method was proposed by Discacciati and Oyarz\'ua
\cite{Discacciati:2017}. They use Bernardi--Raugel and Raviart--Thomas
elements for the velocities, piecewise constants for the pressures,
and continuous piecewise linear elements for the Lagrange-multiplier
used to couple the Navier--Stokes and Darcy equations. Extensions of
this work include a conforming mixed finite element method for the
dual mixed formulations of both the Navier--Stokes (with non-linear
viscosity) and Darcy equations \cite{Caucao:2017}, and a conforming
mixed finite element method for the Navier--Stokes/Darcy--Forchheimer
problem \cite{Caucao:2020}. We further mention that a DG
discretization of the Navier--Stokes and dual-mixed formulation of the
Darcy equation was proposed (but not analyzed) in \cite{Girault:2013}
while analysis and finite element formulations of the transient
Navier--Stokes/Darcy problem is addressed, for example, in
\cite{Cesmelioglu:2008,Cesmelioglu:2009,Chaabane:2017,Chidyagwai:2009,Chidyagwai:2010}.

In this paper, we are particularly interested in strongly conservative
discretizations as they can be shown to be \emph{pressure-robust}
\cite{Linke:2012,Linke:2014}, i.e., the velocity error can be shown to
be independent of the best approximation error of pressure scaled by
the inverse of the viscosity. One approach to obtain a pressure-robust
discretization is by using divergence-conforming velocity spaces
\cite{John:2017}. Divergence-conforming DG methods have been
introduced for the Stokes and Navier--Stokes equations in
\cite{Cockburn:2007a,Wang:2007}. However, DG methods are known to be
expensive. As a remedy, hybridizable discontinuous Galerkin (HDG)
methods were introduced in \cite{Cockburn:2009a} to improve the
computational efficiency of traditional DG methods through
hybridization. Recently, divergence-conforming HDG methods have been
introduced for the Stokes and Navier--Stokes equations, e.g.,
\cite{Cockburn:2014b,Fu:2019,Lehrenfeld:2016,Rhebergen:2018a}.

For the Stokes--Darcy problem a strongly conservative DG
discretization was proposed by Girault, Kanschat and Rivi\`ere
\cite{Girault:2014,Kanschat:2010} using divergence-conforming velocity
spaces. A strongly conservative HDG method, using similar spaces, was
later proposed by Fu and Lehrenfeld \cite{Fu:2018} for the same
problem. We, however, are not aware of divergence-conforming DG or HDG
methods for the coupled Navier--Stokes/Darcy problem.

In \cite{Cesmelioglu:2020}, we presented a strongly conservative
discretization of the velocity-pressure formulation of the Stokes
equations coupled to the dual-mixed formulation of the Darcy
equations. In this paper we extend this approach to the coupled
Navier--Stokes/Darcy problem. The problem is formulated in
\cref{sec:navierstokesdarcy} while existence and uniqueness of a
solution to this problem is discussed in \cref{s:existuniq_wf}. The
HDG method is proposed in \cref{sec:hdgnavierstokesdarcy} where we
also discuss well-posedness of this discretization. We present an a
priori error analysis of the method in \cref{s:errorAnalysis} where we
prove optimal (pressure-robust) rates of convergence in the energy
norm. Numerical examples in \cref{sec:numerical_examples} serve to
verify our theoretical results and conclusions are drawn in
\cref{ss:conclusion}.

\section{The coupled Navier--Stokes and Darcy problem}
\label{sec:navierstokesdarcy}

We consider the coupled Navier--Stokes and Darcy problem on a bounded
two ($\dim=2$) or three ($\dim=3$) dimensional domain
$\Omega \subset \mathbb{R}^{\dim}$ which is decomposed into two
disjoint domains $\Omega^s$ and $\Omega^d$. Fluid flow in the free
fluid region $\Omega^s$ is modeled by the Navier--Stokes equations
while in the porous region $\Omega^d$ fluid flow is modeled by the
Darcy equations. This coupled system of equations is given by
\begin{subequations}
  \label{eq:system}
  \begin{align}
    \label{eq:momentum}
    \nabla \cdot(u \otimes u) + \nabla \cdot \sigma &= f^s && \text{in } \Omega^s,
    \\
    \label{eq:d_velocity}
    \mu\kappa^{-1}u + \nabla p &= 0 & & \text{in } \Omega^d,
    \\
    \label{eq:mass}
    -\nabla\cdot u &= \chi^df^d & & \text{in } \Omega,
  \end{align}
\end{subequations}
where $u$ is the fluid velocity, $p$ denotes the kinematic pressure in
$\Omega^s$ and the piezometric head in $\Omega^d$,
$\sigma := p\mathbb{I} - 2\mu\varepsilon(u)$ is the diffusive part of
the fluid momentum flux in $\Omega^s$,
$\varepsilon(u) := \del[0]{\nabla u + (\nabla u)^T}/2$ is the strain
rate tensor, $\mu > 0$ is the constant kinematic viscosity, $f^s$ a
body force, $f^d$ is a source/sink term, and $\kappa$ is a positive
definite symmetric matrix corresponding to the permeability of
$\Omega^d$. We will assume that there exist positive constants
$\kappa_{\min}$ and $\kappa_{\max}$ such that
\begin{equation*}
  \forall z \in \mathbb{R}^{\dim},\
  \kappa_{\min}|z|^2 \le (\kappa(x) z)\cdot z \le \kappa_{\max}|z|^2, \quad x \in \Omega^d.
\end{equation*}

The boundary of the domain $\Omega$, $\partial \Omega$, is assumed to
be a polyhedral Lipschitz boundary. The boundaries of $\Omega^j$ are
denoted by $\partial\Omega^j$, where $j=s,d$. The interface
$\Gamma^I := \partial \Omega^s \cap \partial \Omega^d$ is assumed to
be Lipschitz polyhedral. We furthermore define the exterior boundaries
$\Gamma^j := \partial\Omega^j \backslash \Gamma^I$ for $j=s,d$. On
$\Gamma^j$, we denote by $n$ the outward unit normal to $\Omega^j$
while on $\Gamma^I$, $n$ denotes the unit normal vector pointing
outward from $\Omega^s$. On $\Gamma^I$ we further define the
orthonormal tangential vectors $\tau^k$ for $1 \le k \le \dim-1$.

Let $\chi^d$ be the characteristic function that has the value 1 in
$\Omega^d$ and 0 in $\Omega^s$ and let $\chi^s = 1-\chi^d$. We define
$u^j = \chi^ju$ and $p^j=\chi^jp$ for $j=s,d$. The Navier--Stokes and
Darcy equations are coupled at the interface by assuming continuity of
the normal component of the velocity, the Beavers--Joseph--Saffman law
\cite{Beavers:1967, Saffman:1971}, and a balance of forces. These
assumptions result in the following transmission conditions on
$\Gamma^I$:
\begin{subequations}
  \label{eq:interface}
  \begin{align}
    \label{eq:bc_I_u}
    u^s\cdot n &= u^d \cdot n & & \text{on } \Gamma^I,
    \\
    \label{eq:bc_I_slip}
    -2\mu\del[0]{\varepsilon(u^s)n} \cdot \tau^i &= \frac{\alpha \mu}{\sqrt{\kappa_i}} u^s \cdot \tau^i
                                                   \quad,\ 1\le i \le \dim -1
                              & & \text{on } \Gamma^I,
    \\
    \label{eq:bc_I_p}
    (\sigma n)\cdot n &= p^d & & \text{on } \Gamma^I,
  \end{align}
\end{subequations}
where $\kappa_i = \tau^i \cdot (\kappa \tau^i)$ and $\alpha > 0$ is an
experimentally determined dimensionless constant. To complete the
problem description, we impose the following exterior boundary
conditions:
\begin{equation}
  \label{eq:bcs_simple}
  u = 0 \text{ on } \Gamma^s
  \quad \text{ and } \quad
  u\cdot n = 0 \text{ on } \Gamma^d.  
\end{equation}
Finally, for well-posedness of the problem, we require
$\int_{\Omega^d}f^d\dif x=0$ and $\int_{\Omega}p \dif x = 0$.

\section{Existence and uniqueness of solutions to a mixed weak formulation}
\label{s:existuniq_wf}

Let
\begin{equation}
  \label{eq:functionspaceX}
  \mathcal{X} := \cbr[0]{ u = (u^s, u^d) \in \mathcal{X}^s \times \mathcal{X}^d:\ u^s \cdot n = u^d \cdot n \text{ on } \Gamma^I},
\end{equation}
where
$\mathcal{X}^s := \cbr[0]{v \in \sbr[0]{H^1(\Omega^s)}^{\dim}:\ v=0
  \text{ on } \Gamma^s}$,
$\mathcal{X}^d := \cbr[0]{v \in H(\text{div};\Omega^d):\ v \cdot n = 0
  \text{ on } \Gamma^d}$, and endow $\mathcal{X}$ with the product
norm
$\norm{u}_{\mathcal{X}} := ( \envert[0]{u^s}_{1,\Omega^s}^2 +
\norm[0]{u^d}_{\text{div};\Omega^d}^2)^{1/2}$ for all
$u \equiv (u^s, u^d) \in \mathcal{X}$, where
$\norm[0]{u^d}_{\text{div};\Omega^d}^2 := \norm[0]{u^d}_{\Omega^d}^2 +
\norm[0]{\nabla \cdot u^d}_{\Omega^d}^2$. Since the
$\norm[0]{\cdot}_{\text{div};\Omega^d}$-norm is only applied to
functions on $\Omega^d$, we will drop the subscript $\Omega^d$ and
write $\norm[0]{\cdot}_{\text{div}}$.

The following mixed weak formulation for
\cref{eq:system,eq:interface,eq:bcs_simple} was proposed in
\cite[Section 3.3]{Girault:2013}: Find
$(u,p) \in \mathcal{X} \times \mathcal{Q}$, with
$\mathcal{Q}:=L_0^2(\Omega)$, such that:
\begin{equation}
  \label{eq:weakform}
  a(u; u,v) + b(v,p) + b(u,q) = \ell^s(v) + \ell^d(q)
  \quad \forall (v, q) \in \mathcal{X} \times \mathcal{Q},
\end{equation}
where $\ell^s(v) := \int_{\Omega^s}f^s \cdot v \dif x$,
$\ell^d(q) := \int_{\Omega^d} f^d q \dif x$, and where the forms
$a: \mathcal{X} \times \mathcal{X} \times \mathcal{X} \to \mathbb{R}$
and $b: \mathcal{X} \times \mathcal{Q} \to \mathbb{R}$ are defined as:
\begin{equation*}
  a(w; u, v)
  := t(w; u^s, v^s) + a^s(u^s, v^s) + a^d(u^d, v^d) + a^I(u^s, u^s),
  \quad
  b(v, q)
  := - \int_{\Omega} q \nabla \cdot v \dif x,  
\end{equation*}
with
\begin{align*}
  a^s(u, v)
  &:= \int_{\Omega^s} 2\mu \varepsilon(u) : \varepsilon(v) \dif x,
  & a^d(u, v)
  &:= \int_{\Omega^d} \mu\kappa^{-1} u \cdot v \dif x,
  \\
  a^I(u, v)
  &:= \int_{\Gamma^I} \sum_{i=1}^{\dim-1} \frac{\alpha \mu}{\sqrt{\kappa_i}} (u\cdot \tau^i) (v \cdot \tau^i) \dif s,
  & t(w;u,v)
  &:= \int_{\Omega^s} (w \cdot \nabla u) \cdot v \dif x.
\end{align*}  

Well-posedness for the coupled Navier--Stokes and Darcy problem was
shown in \cite{Discacciati:2017} for the case that the source/sink
term $f^d$ in \cref{eq:mass} is zero. With minor modifications of
these proofs, well-posedness can be shown also for the case
$f^d \ne 0$. In particular, it can be shown that if the data
$f^s \in \sbr[0]{L^2(\Omega^s)}^{\dim}$ and $f^d \in L^2(\Omega^d)$
satisfy the smallness condition
\begin{equation}
  \label{eq:uniquenessrequirement-wf}
  C_p\norm[0]{f^s}_{\Omega^s} + 2\mu C_fC_{bb}^{-1}\norm[0]{f^d}_{\Omega^d}
  < \mu^2C_{ae} \min\del{C_{ae}C_w^{-1}, 2 C_{ae}^s\delta C_{si,2}^{-1}C_{si,4}^{-2}},
\end{equation}
then there exists a unique solution
$(u,p) \in \mathcal{X} \times \mathcal{Q}$ to
\cref{eq:weakform}. Moreover, this solution satisfies
\begin{subequations}
  \label{eq:bounduXpQ-wf}
  \begin{align}
    \label{eq:bounduXpQ-wf-a}
    \norm{u}_{\mathcal{X}}
    &\le C_{ae}^{-1}\del[1]{ \mu^{-1}C_p\norm[0]{f^s}_{\Omega^s} + 2 C_{f}C_{bb}^{-1}\norm[0]{f^d}_{\Omega^d} },
    \\
    \label{eq:bounduXpQ-wf-b}
    \norm{p}_{\Omega}
    &\le 2C_{f}(C_{ae}C_{bb})^{-1}\del[1]{ C_p\norm[0]{f^s}_{\Omega^s} + C_{f}C_{bb}^{-1}\mu \norm[0]{f^d}_{\Omega^d} }.
  \end{align}
\end{subequations}
We remark that $C_f$ and $C_{ae}$ are the constants related to,
respectively, the boundedness and coercivity of
$a(\cdot;\cdot,\cdot)$, which are given by:
\begin{equation*}
  \begin{split}
    C_{f} &= \max\del[0]{2 + C_{ac}^I\alpha \kappa_{\min}^{-1/2} + 2C_w
      C_{ae}^s\delta C_{si,2}^{-1}C_{si,4}^{-2}, \kappa_{\min}^{-1}},
    \\
    C_{ae} & = \min(C_{ae}^s(1-\delta), \kappa_{\max}^{-1}).
  \end{split}
\end{equation*}
All other constants are independent of $\kappa_{\min}$,
$\kappa_{\max}$, $\alpha$, and $\mu$. Here $C_{ae}^s$ and $C_{ac}^I$
are related to, respectively, the coercivity constant of $a^s$ and the
continuity constant of $a^I$. Furthermore, $C_{bb}$ is the inf-sup
constant of $b(\cdot, \cdot)$, $C_p$ is the Poincar\'e constant, $C_w$
is a constant related to the dimension of the problem and the Sobolev
embedding constant from $H^1(\Omega^s)$ into $L^4(\Omega^s)$, $\delta$
is a constant that lies in $(0,1)$, and $C_{si,2}$ and $C_{si,4}$ are
constants relating, respectively, the $L^2$- and $L^4$-norms on the
interface to the $H^1$-norm on $\Omega^s$.

\section{The hybridizable discontinuous Galerkin method}
\label{sec:hdgnavierstokesdarcy}

\subsection{The discretization}
\label{ss:discretization}

To define the discretization we first introduce the triangulations
$\mathcal{T}^j$ of $\Omega_j$, with $j=s,d$. We assume these
triangulations consist of shape-regular simplices $K$ and that the two
triangulations $\mathcal{T}^j$, $j=s,d$, coincide on the interface
$\Gamma^I$. We further denote by
$\mathcal{T} := \mathcal{T}^s \cup \mathcal{T}^d$ the triangulation of
$\Omega$ and we define $h:= \max_{K\in\mathcal{T}}h_K$, where $h_K$ is
the diameter of $K$.

On the cells $K$ we define the discontinuous finite element spaces
\begin{equation*}
  \label{eq:DGFEMspacesdef}
  \begin{split}
    X_h &:= \cbr[1]{v_h \in \sbr[0]{L^2(\Omega)}^{\dim} : \ v_h \in
      \sbr[0]{P_k(K)}^{\dim}, \ \forall\ K \in \mathcal{T}},
    \\
    X_h^j &:= \cbr[1]{v_h \in \sbr[0]{L^2(\Omega^j)}^{\dim} : \ v_h \in
      \sbr[0]{P_k(K)}^{\dim}, \ \forall\ K \in \mathcal{T}^j}, \quad j=s,d,
    \\
    Q_h &:= \cbr[1]{q_h \in L^2_0(\Omega) : \ q_h \in P_{k-1}(K) ,\
      \forall \ K \in \mathcal{T}},
    \\
    Q_h^j &:= \cbr[1]{q_h \in Q_h : \ q_h \in P_{k-1}(K) ,\
      \forall \ K \in \mathcal{T}^j}, \quad j=s,d,
  \end{split}
\end{equation*}
where $P_l(K)$ denotes the space of polynomials of degree $l$ on any
cell $K$.

By $\mathcal{F}^j$ and $\Gamma_0^j$ we denote the set and union of
facets $F$ on the subdomain $\overline{\Omega}^j$, $j=s,d$. By
$\mathcal{F}$ and $\Gamma_0$ we denote the set and union of all facets
in $\overline{\Omega}$ while $\mathcal{F}^I$ denotes the set of all
facets on $\Gamma^I$. Then, denoting by $P_m(F)$ the space of
polynomials of degree $m$ on any facet $F$, we define the following
facet finite element spaces:
\begin{equation*}
  \label{eq:HDGFEMspacesdef}
  \begin{split}
    \bar{X}_h &:= \cbr[1]{\bar{v}_h \in \sbr[0]{L^2(\Gamma_0^s)}^{\dim}:\
      \bar{v}_h \in \sbr[0]{P_{k}(F)}^{\dim}\ \forall\ F \in \mathcal{F}^s,\
      \bar{v}_h = 0 \text{ on } \Gamma^s},
    \\
    \bar{Q}_h^j &:= \cbr[1]{\bar{q}_h^j \in L^2(\Gamma_0^j) : \ \bar{q}_h^j
      \in P_{k}(F) \ \forall\ F \in \mathcal{F}^j},\quad j=s,d.    
  \end{split}
\end{equation*}
Grouping the cell and facet unknowns in the following compact
notation:
\begin{align*}
  \boldsymbol{v}_h
  &:= (v_h, \bar{v}_h) \in \boldsymbol{X}_h := X_h \times \bar{X}_h,
  &
  \boldsymbol{v}_h^s
  &:= (v_h^s, \bar{v}_h) \in \boldsymbol{X}_h^s := X_h^s \times \bar{X}_h,
  \\
  \boldsymbol{q}_h^{j}
  &:= (q_h, \bar{q}_h^j) \in \boldsymbol{Q}_h^{j} := Q_h^j \times \bar{Q}_h^j,\ j=s,d,
  &
  \boldsymbol{q}_h
  &:= (q_h, \bar{q}_h^s, \bar{q}_h^d) \in \boldsymbol{Q}_h := Q_h \times \bar{Q}_h^s \times \bar{Q}_h^d,
\end{align*}
we propose the following HDG discretization for
\cref{eq:system,eq:interface,eq:bcs_simple}: Find
$(\boldsymbol{u}_h, \boldsymbol{p}_h) \in \boldsymbol{X}_h \times
\boldsymbol{Q}_h$ such that for all
$(\boldsymbol{v}_h, \boldsymbol{q}_h) \in \boldsymbol{X}_h \times
\boldsymbol{Q}_h$
\begin{equation}
  \label{eq:hdg}
  a_h(u_h; \boldsymbol{u}_h, \boldsymbol{v}_h) + b_h(\boldsymbol{v}_h, \boldsymbol{p}_h) + b_h(\boldsymbol{u}_h, \boldsymbol{q}_h)
  = \ell^s(v_h) + \ell^d(q_h),  
\end{equation}
where
\begin{subequations}
  \begin{align}
    \label{eq:def_ah}
    a_h(w; \boldsymbol{u}, \boldsymbol{v})
    :=& t_h(w; \boldsymbol{u}, \boldsymbol{v})
        + a_h^s(\boldsymbol{u}, \boldsymbol{v})
        + a^d(u, v) + a^I(\bar{u}, \bar{v}),
    \\
    \label{eq:ah_s}
    a_h^s(\boldsymbol{u}, \boldsymbol{v})
    :=&
        \sum_{K\in\mathcal{T}^s} \int_K 2\mu \varepsilon(u) : \varepsilon(v) \dif x
        +\sum_{K\in\mathcal{T}^s} \int_{\partial K} \frac{2\beta\mu}{h_K}(u-\bar{u}) \cdot (v-\bar{v}) \dif s
    \\
    \nonumber
      &-\sum_{K\in\mathcal{T}^s} \int_{\partial K} 2\mu\varepsilon(u)n \cdot (v-\bar{v})\dif s
        -\sum_{K\in\mathcal{T}^s} \int_{\partial K} 2\mu\varepsilon(v)n \cdot (u-\bar{u})\dif s,
    \\
    \label{eq:oh}
    t_h(w; \boldsymbol{u}, \boldsymbol{v})
    :=&
        -\sum_{K\in\mathcal{T}^s}\int_{K} u\otimes w : \nabla v \dif x
        +\sum_{K\in\mathcal{T}^s}\int_{\partial K}\tfrac{1}{2}w\cdot n \, (u+\bar{u})\cdot(v-\bar{v}) \dif s
    \\ \nonumber
      &+\sum_{K\in\mathcal{T}^s}\int_{\partial K}\tfrac{1}{2}\envert{w\cdot n}(u-\bar{u})\cdot(v-\bar{v})\dif s
        +\int_{\Gamma^I}(w\cdot n)\bar{u}\cdot \bar{v} \dif s,   
  \end{align}
\end{subequations}
and where
\begin{subequations}
  \begin{align*}
    b_h(\boldsymbol{v}, \boldsymbol{q})
    :=& b_h^s(v, \boldsymbol{q}^s) + b_h^{I,s}(\bar{v}, \bar{q}^s) + b_h^d(v, \boldsymbol{q}^d) + b_h^{I,d}(\bar{v}, \bar{q}^d),
    \\
    b_h^{j}(\boldsymbol{p}^{j},v )
    :=& -\sum_{K\in\mathcal{T}^j} \int_K p \nabla\cdot v \dif x
              + \sum_{K\in\mathcal{T}^j} \int_{\partial K} \bar{p}^j v\cdot n^j \dif s,    
    \qquad
    b_h^{I,j}(\bar{p}^j, \bar{v} )
      := -\int_{\Gamma^I}\bar{p}^j\bar{v}\cdot n^j \dif s,    
  \end{align*}
\end{subequations}
for $j=s,d$. In the above definitions, $\beta > 0$ is a penalty
parameter and $n^j$ is the outward unit normal vector on the boundary
of any element $K \in \mathcal{T}^j$. On the interface $\Gamma^I$,
$n^s = -n^d$. If it is clear to which set $K$ belongs, we drop the
superscript $j$.

The HDG discretization \cref{eq:hdg} is the Navier--Stokes/Darcy
extension of the discretization recently proposed for the coupled
Stokes/Darcy problem in \cite{Cesmelioglu:2020} (where the matrix
$\kappa$ corresponding to the permeability was replaced by a positive
constant). This discretization is strongly conservative, a property
inherited by the HDG discretization of the Navier--Stokes/Darcy
discretization. Indeed, the velocity solution $u_h \in X_h$ to
\cref{eq:hdg} satisfies:
\begin{subequations}
  \label{eq:massconservations}
  \begin{align}
    \label{eq:massconservation-1}
    -\nabla \cdot u_h &= \chi^d\Pi_Qf^d && \forall x\in K, \ \forall K \in \mathcal{T},
    \\
    \label{eq:massconservation-2}
    \jump{u_h\cdot n}&=0 && \forall x\in F,\ \forall F\in \mathcal{F}, 
    \\
    \label{eq:massconservation-4}
    u_h\cdot n&=\bar{u}_h\cdot n && \forall x\in F,\ \forall F\in \mathcal{F}^I,
  \end{align}
\end{subequations}
where $\Pi_Q$ is the $L^2$-projection operator into $Q_h$,
$\jump{\cdot}$ is the usual jump operator, and $n$ is the unit normal
vector on~$F$. See \cite[Section 3.3]{Cesmelioglu:2020} for a proof of
\cref{eq:massconservations}. In the following analysis it will be
useful to introduce the following subspaces:
\begin{subequations}
  \begin{align*}
    \boldsymbol{Z}_h^s
    :&= \cbr[1]{ \boldsymbol{v}_h \in \boldsymbol{X}_h^s:\ b_h^s(v_h,\boldsymbol{q}^s_h) + b_h^{I,s}(\bar{v}_h,\bar{q}_h^s)
      = 0 \ \forall \boldsymbol{q}_h^s \in \boldsymbol{Q}_h^s},
    \\
    \boldsymbol{Z}_h
    :&= \cbr[1]{ \boldsymbol{v}_h \in \boldsymbol{X}_h:\
      \sum_{j=s,d}\del[0]{b_h^j(v_h,\boldsymbol{q}^j_h) + b_h^{I,j}(\bar{v}_h,\bar{q}_h^j)}
      = 0 \ \forall \boldsymbol{q}_h \in \boldsymbol{Q}_h}.
  \end{align*}
\end{subequations}
The velocity solution to \cref{eq:hdg}, $\boldsymbol{u}_h$, restricted
to $\Omega^s$ is such that
$\boldsymbol{u}_h^s \in \boldsymbol{Z}_h^s$. Generally, local momentum
conservation of the Navier--Stokes equations needs to be sacrificed to
obtain a stable discretization \cite{Cockburn:2004b}. However, since
$u_h^s$ is exactly divergence-free and $H({\rm div})$-conforming in
$\Omega^s$, this sacrifice is unnecessary and the discretization of
the Navier--Stokes equations in \cref{eq:hdg} is locally momentum
conserving (unlike, for example, the DG method of
\cite{Girault:2009}). Note further that any
$\boldsymbol{v}_h \in \boldsymbol{Z}_h$ satisfies
\cref{eq:massconservations} with $f^d=0$.

\subsection{Notation and extension of known results}
\label{ss:prelim}

Before addressing the well-posedness of the HDG method \cref{eq:hdg},
we briefly introduce notation and extend a few properties of the
discretization previously shown for the coupled Stokes/Darcy problem
in \cite{Cesmelioglu:2020}. On a domain $D$ we will use standard
definitions and notation of the Sobolev spaces $W_p^k(D)$ with
corresponding norms $\norm[0]{\cdot}_{p,k,D}$ (see, for example,
\cite{Adams:book,brenner:book}). If $p=2$ we set $H^k(D) = W_2^k(D)$
and write $\norm[0]{\cdot}_{k,D}$ instead of
$\norm[0]{\cdot}_{2,k,D}$. If $k=0$, we note that $W_p^0(D)$ coincides
with $L^p(D)$. For $p\ne 2$, we denote the norm on $L^p(D)$ by
$\norm{\cdot}_{p,0,D}$. If $p=2$, we denote the norm on $L^2(D)$ by
$\norm{\cdot}_D$.

Let us now introduce the following spaces:
\begin{equation*}
  \begin{split}
    X &:=\cbr[0]{ u = (u^s, u^d) \in X^s \times X^d:\ u^s \cdot n = u^d \cdot n \text{ on } \Gamma^I},
     \\
    Q &:= \cbr[0]{ q \equiv (q^s, q^d):\ q^s \in Q^s:=H^1(\Omega^s),\ q^d \in Q^d:=H^2(\Omega^d),\ \int_{\Omega}q\dif x = 0},
  \end{split}
\end{equation*}
where
$X^s := \cbr[0]{v \in \sbr[0]{H^2(\Omega^s)}^{\dim}:\ v=0 \text{ on }
  \Gamma^s}$ and
$X^d := \cbr[0]{v \in \sbr[0]{H^1(\Omega^d)}^{\dim}:\ v \cdot n = 0
  \text{ on } \Gamma^d}$. We will denote the trace space of $X$ by
$\bar{X}$ and we introduce the trace operator
$\gamma_V: X \to \bar{X}$ which restricts functions in $X$ to
$\Gamma_0^s$. Likewise, the trace space of $Q^j$ is denoted by
$\bar{Q}^j$ and $\gamma_{Q^j} : Q^j \to \bar{Q}^j$ restricts functions
in $Q^j$ to $\Gamma_0^j$, $j=s,d$. If it is clear from the context on
which function spaces the trace operator acts, we drop the subscript
notation from $\gamma$. Using similar notation as in
\cref{ss:discretization}, we next define
$\boldsymbol{X} := X \times \bar{X}$ and
$\boldsymbol{Q} := Q \times \bar{Q}^s \times \bar{Q}^d$. With these
definitions we then introduce the extended function spaces:
\begin{equation*}
  X(h) := X_h + X, \quad
  Q(h) := Q_h + Q, \quad
  \boldsymbol{X}(h) := \boldsymbol{X}_h + \boldsymbol{X}, \quad
  \boldsymbol{Q}(h) := \boldsymbol{Q}_h + \boldsymbol{Q}.    
\end{equation*}
Furthermore, we will need $X^s(h) := X_h^s + X^s$.

\begin{lemma}[Consistency]
  \label{lem:consistency}
  Let $(u, p) \in X \times Q$ solve the Navier--Stokes/Darcy problem
  \cref{eq:system,eq:interface,eq:bcs_simple}. Let
  $\boldsymbol{u}=(u,\gamma(u))$,
  $\boldsymbol{p}=(u,\gamma(p^s),\gamma(p^d))$, then
  \begin{equation*}
    a_h(u; \boldsymbol{u}, \boldsymbol{v}_h) + b_h(\boldsymbol{v}_h, \boldsymbol{p})
    + b_h(\boldsymbol{u}, \boldsymbol{q}_h)
    = \ell^s(v_h) + \ell^d(q_h) \quad
    \forall (\boldsymbol{v}_h, \boldsymbol{q}_h) \in \boldsymbol{X}_h \times \boldsymbol{Q}_h.
  \end{equation*}
\end{lemma}
\begin{proof}
  This is an immediate consequence of \cite[Lemma 1]{Cesmelioglu:2020}
  and that
  \begin{equation*}
    t_h(u; \boldsymbol{u}, \boldsymbol{v}_h)
    =
    \sum_{K\in\mathcal{T}^s}\int_{K} \nabla \cdot (u\otimes u) \cdot v_h \dif x \quad \forall \boldsymbol{v}_h \in \boldsymbol{X}_h,
  \end{equation*}
  which follows by integration by parts, smoothness of $u$, and
  single-valuedness of $\bar{v}_h$. \qed
\end{proof}

For the analysis of \cref{eq:hdg} we require the following two
norms defined on $\boldsymbol{X}(h)$:
\begin{align*}
  \tnorm{\boldsymbol{v}}_{v}^2
  :&=
     \tnorm{\boldsymbol{v}}_{v,s}^2
     + \tnorm{\boldsymbol{v}}_{v,d}^2
     + \norm[0]{ \bar{v}^t }^2_{\Gamma^I},
  \\
  \tnorm{\boldsymbol{v}}_{v'}^2
  :&= \tnorm{\boldsymbol{v}}_{v}^2
     + \sum_{K\in \mathcal{T}^s} h_K^2 \envert{v}_{2,K}^2
     =\tnorm{\boldsymbol{v}}_{v',s}^2  
     + \tnorm{\boldsymbol{v}}_{v,d}^2
     + \norm[0]{ \bar{v}^t }^2_{\Gamma^I},
\end{align*}
where
\begin{equation*}
  \begin{split}
    \tnorm{\boldsymbol{v}}_{v,s}^2
    &:= \sum_{K\in \mathcal{T}^s}\del[1]{\norm[0]{\nabla v}_K^2
      + h_K^{-1}\norm{v-\bar{v}}^2_{\partial K}},
    \\
    \tnorm{\boldsymbol{v}}_{v',s}^2
    &:= \tnorm{\boldsymbol{v}}_{v,s}^2
    + \sum_{K\in \mathcal{T}^s} h_K^2 \envert{v}_{2,K}^2,
    \\
    \tnorm{\boldsymbol{v}}_{v,d}^2
    &:= \norm{v}_{\text{div}}^2 + \sum_{F \in \mathcal{F}^d\backslash\mathcal{F}^I}h_F^{-1}\norm[0]{\jump{v\cdot n}}_{F}^2
    + \sum_{K \in \mathcal{T}^d} h_K^{-1}\norm{(v - \bar{v})\cdot n}_{\partial K \cap \Gamma^I}^2,    
  \end{split}  
\end{equation*}
where we remark that $\jump{v \cdot n} = v \cdot n$ on
$\Gamma^d$. 

\begin{remark}
  \label{rem:different_norm}
  The norms $\tnorm{\cdot}_v$ and $\tnorm{\cdot}_{v'}$ are different
  from the norms used in \cite{Cesmelioglu:2020}; instead of only
  $\norm[0]{v}_{\Omega^d}$ in $\tnorm{\boldsymbol{v}}_{v}$ we use
  $\tnorm{\boldsymbol{v}}_{v,d}$ for the functions in the Darcy part
  of the domain.
\end{remark}

The norms $\tnorm{\cdot}_v$ and $\tnorm{\cdot}_{v'}$ are equivalent on
$\boldsymbol{X}_h$, i.e., there exists a constant $c_e$ such that
$\tnorm{\boldsymbol{v}}_v \le \tnorm{\boldsymbol{v}}_{v'} \le c_e
\tnorm{\boldsymbol{v}}_v$ for all
$\boldsymbol{v} \in \boldsymbol{X}_h$, see
\cite[eq.~(5.5)]{Wells:2011}. On $X_h^s$, we will also require the
following results from \cite[Theorem 4.4 and Proposition
4.5]{Girault:2009}:
\begin{align}
  \label{eq:dpoincareineq}
  \norm{v_h}_{\Omega^s}
  &\le c_{p} \norm{v_h}_{1,h,\Omega^s}
    \le c_{p} \tnorm{\boldsymbol{v}_h}_{v,s}
    \quad \forall \boldsymbol{v}_h \in \boldsymbol{X}_h^s,
  \\
  \label{eq:dtrpoincareineq}
  \text{For } r \ge 2: \quad
  \norm{v_h^s}_{r,0,\Gamma^I}
  &\le c_{si,r} \norm{v_h}_{1,h,\Omega^s}
    \le c_{si,r} \tnorm{\boldsymbol{v}_h}_{v,s}
    \quad \forall \boldsymbol{v}_h \in \boldsymbol{X}_h^s,  
\end{align}
where $\norm{v_h}_{1,h,\Omega^s} := \tnorm{(v_h, \av{v_h})}_{v,s}$ and
where $c_p$ and $c_{si,r}$ are positive constants independent of
$h$. On $\boldsymbol{Q}(h)$, we define
\begin{equation*}
  \tnorm{\boldsymbol{q}}_p^2 := \tnorm{\boldsymbol{q}^s}_{p,s}^2 + \tnorm{\boldsymbol{q}^d}_{p,d}^2
  \quad\text{where}\quad
  \tnorm{\boldsymbol{q}^j}_{p,j}^2 := \norm{q}_{\Omega^j}^2 + \sum_{K
  \in \mathcal{T}^j} h_K \norm[0]{\bar{q}^j}_{\partial K}^2,\ j=s,d.
\end{equation*}

For the linear forms on the right hand side of \cref{eq:hdg} we note,
using \cref{eq:dpoincareineq}, that
\begin{subequations}
  \label{eq:boundsfsfd}
  \begin{align}
    |\ell^s(v_h)|
    &\le \norm[0]{f^s}_{\Omega^s}\norm[0]{v_h}_{\Omega^s} \le c_p \norm[0]{f^s}_{\Omega^s}\tnorm{\boldsymbol{v}_h}_{v,s}
      \le c_p \norm[0]{f^s}_{\Omega^s}\tnorm{\boldsymbol{v}_h}_{v} && \forall \boldsymbol{v}_h \in \boldsymbol{X}_h,
    \\
    |\ell^d(q_h)|
    &\le \norm[0]{f^d}_{\Omega^d}\norm[0]{q_h}_{\Omega^d} \le \norm[0]{f^d}_{\Omega^d}\tnorm{\boldsymbol{q}_h}_{p,d}
      \le \norm[0]{f^d}_{\Omega^d}\tnorm{\boldsymbol{q}_h}_{p} && \forall \boldsymbol{q}_h \in \boldsymbol{Q}_h.
  \end{align}  
\end{subequations}

The following properties of the different bilinear forms in
\cref{eq:hdg} hold (see \cite[Lemma 3]{Cesmelioglu:2020}): for all
$\boldsymbol{u}, \boldsymbol{v} \in \boldsymbol{X}(h)$:
\begin{subequations}
  \label{eq:continuity-ah-sep-terms}
  \begin{align}
    |a_h^s(\boldsymbol{u}, \boldsymbol{v})|
    &\le \mu c_{ac}^s \tnorm{\boldsymbol{u}}_{v',s}\tnorm{\boldsymbol{v}}_{v',s},
    \\
    |a^d(u, v)|
    &\le \mu\kappa_{\min}^{-1} \norm{u}_{\Omega^d}\norm{v}_{\Omega^d},
    \\
    |a^I(\bar{u}, \bar{v})|
    &\le \alpha\mu\kappa_{\min}^{-1/2}\norm[0]{\bar{u}^t}_{\Gamma^I}\norm[0]{\bar{v}^t}_{\Gamma^I},
  \end{align}
\end{subequations}
where $c_{ac}^s>0$ is a constant independent of $h$ and
$(z)^t = z - (z\cdot n)n$. Furthermore, by \cite[Lemma
2]{Cesmelioglu:2020}, there exists a constant $c_{ae}^s>0$,
independent of $h$, and a constant $\beta_0 > 0$ such that for
$\beta > \beta_0$ and for all $\boldsymbol{v}_h \in \boldsymbol{X}_h$:
\begin{subequations}
  \label{eq:coercivity_ahsdI}
  \begin{equation}
    a_h^s(\boldsymbol{v}_h, \boldsymbol{v}_h)
    \ge \mu c_{ae}^s \tnorm{\boldsymbol{v}_h}_{v,s}^2.
  \end{equation}
  Additionally, for all $\boldsymbol{v}_h \in \boldsymbol{X}_h$,
  \begin{equation}
    a^d(v_h, v_h)
    \ge \mu \kappa_{\max}^{-1} \norm{v_h}_{\Omega^d}^2,
    \quad
    a^I(\bar{v}_h, \bar{v}_h)
    \ge \alpha\mu\kappa_{\max}^{-1/2}\norm[0]{\bar{v}_h^t}_{\Gamma^I}^2.    
  \end{equation}
\end{subequations}

To prove properties of the bilinear form $b_h$, we use the
Brezzi--Douglas--Marini (BDM) interpolation operator
$\Pi_V : \sbr[0]{H^1(\Omega)}^{\dim} \to X_h \cap
H(\text{div};\Omega)$. For all $u \in \sbr[0]{H^{k+1}(K)}^{\dim}$ this
interpolation operator satisfies \cite[Lemma 7]{Hansbo:2002}:
\begin{subequations}
  \label{eq:PiVinterpOp}
  \begin{align}
    \label{eq:div_PiV}
    \int_K q_h (\nabla \cdot u - \nabla \cdot \Pi_V u) \dif x
    &= 0 && \forall q_h \in P_{k-1}(K),
    \\
    \label{eq:jump_PiV}
    \int_F \bar{q}_h( n \cdot u - n \cdot \Pi_Vu ) \dif s
    &= 0 && \forall \bar{q}_h \in P_k(F),\ F \text{ is a face on } \partial K,
  \end{align}
\end{subequations}
as well as the interpolation estimates:
\begin{subequations}
  \label{eq:PiVintest}
  \begin{align}
    \label{eq:PiVintest-u}
    \norm[0]{u - \Pi_Vu}_{m,K} &\le ch_K^{l-m}\norm[0]{u}_{l,K}, && m=0,1,2,\ \max(1,m) \le l \le k+1,
    \\
    \label{eq:PiVintest-divu}
    \norm[0]{\nabla \cdot (u - \Pi_Vu)}_{K} &\le c h_K^{l}\norm[0]{\nabla \cdot u}_{l,K}, && 0 \le l \le k.
  \end{align}
\end{subequations}
We will also require the $L^2$-projection into the facet velocity
space, $\bar{\Pi}_V:\sbr[0]{H^1(\Omega^s)}^{\dim} \to \bar{X}_h$. The
following two lemmas prove properties of the bilinear form $b_h$. We
start in \cref{lem:infsupbh} by proving an inf-sup condition. An
inf-sup condition for $b_h$ was recently proven in \cite[Theorem
1]{Cesmelioglu:2020}. However, this proof is modified here due to the
use of different norms (see \cref{rem:different_norm}).

\begin{lemma}
  \label{lem:infsupbh}
  There exists a positive constant $c_{bb}$, independent of $h$, such
  that for all $\boldsymbol{q}_h \in \boldsymbol{Q}_h$,
  \begin{equation}
    \label{eq:infsupbh}
    c_{bb} \tnorm{\boldsymbol{q}_h}_p \le
    \sup_{\substack{\boldsymbol{v}_h \in \boldsymbol{X}_h \\ \boldsymbol{v}_h \ne 0}}
    \frac{b_h(\boldsymbol{v}_h, \boldsymbol{q}_h)}{\tnorm{\boldsymbol{v}_h}_{v}}.
  \end{equation}
\end{lemma}
\begin{proof}
  Let us write
  $b_h(\boldsymbol{v}_h, \boldsymbol{q}_h) = b_h^{1}(v_h, q_h) +
  b_h^{2}(\boldsymbol{v}_h, (\bar{q}^s_h, \bar{q}^d_h))$, where
  \begin{align*}
    b_h^{1}(v_h, q_h)
    &= -\sum_{j=s,d}\sum_{K\in\mathcal{T}^j} \int_K q_h \nabla\cdot v_h \dif x,
    \\
    b_h^{2}(\boldsymbol{v}_h, (\bar{q}^s_h,\bar{q}^d_h))
    &= \sum_{j=s,d}\del[2]{\sum_{K\in\mathcal{T}^j} \int_{\partial K} \bar{q}^j_h v_h\cdot n^j \dif s - \int_{\Gamma^I}\bar{q}_h^j\bar{v}_h\cdot n^j \dif s}.
  \end{align*}
  To prove \cref{eq:infsupbh} we apply \cite[Theorem 3.1]{Howell:2011}
  which, in this context, states that if there exist constants
  $c_{b1}>0$ and $c_{b2}>0$, independent of $h$, such that for all
  $q_h \in Q_h$ and
  $(\bar{q}_h^s,\bar{q}_h^d) \in \bar{Q}_h^s \times \bar{Q}_h^d$,
  \begin{subequations}
    \label{eq:infsup-b1-b2}
    \begin{align}
      \label{eq:infsup-b1}
      c_{b1} \norm[0]{q_h}_{\Omega}
      &\le \sup_{\substack{\boldsymbol{v}_h \in \text{Ker}(b_h^2) \\ \boldsymbol{v}_h \ne 0}}
      \frac{ b_h^{1}(v_h,q_h)}{\tnorm{\boldsymbol{v}_h}_v},
      \\
      \label{eq:infsup-b2}
      \del[2]{c_{b2}\sum_{j=s,d}\sum_{K\in\mathcal{T}^j}h_K\norm[0]{\bar{q}_h^j}^2_{\partial K}}^{1/2}      
      &\le \sup_{\substack{\boldsymbol{v}_h \in \boldsymbol{X}_h \\ \boldsymbol{v}_h \ne 0}}
      \frac{ b_h^{2}(\boldsymbol{v}_h, (\bar{q}^s_h,\bar{q}^d_h))}{\tnorm{\boldsymbol{v}_h}_v},
    \end{align} 
  \end{subequations}
  where
  $\text{Ker}(b_h^2) := \cbr[0]{\boldsymbol{v}_h \in
    \boldsymbol{X}_h\,:\, b_h^{2}(\boldsymbol{v}_h, (\bar{q}^s_h,
    \bar{q}^d_h)) = 0 \ \forall (\bar{q}_h^s,\bar{q}_h^d) \in
    \bar{Q}_h^s \times \bar{Q}_h^d}$, then there exists a constant
  $c_{bb}$ independent of $h$ such that \cref{eq:infsupbh} holds for
  all $\boldsymbol{q}_h \in \boldsymbol{Q}_h$. As such, we now
  separately prove \cref{eq:infsup-b1,eq:infsup-b2}.

  We start with \cref{eq:infsup-b1}. Let $q_h \in Q_h$. Then, since
  $q_h \in L^2_0(\Omega)$ there exists a
  $v \in \sbr[0]{H^1_0(\Omega)}^{\dim}$ such that
  $-\nabla \cdot v = q_h$ and
  $c_{vq}\norm{v}_{1,\Omega} \le \norm[0]{q_h}_{\Omega}$ (e.g,
  \cite[Theorem 6.5]{Pietro:book}). By properties of $\Pi_V$ and
  $\bar{\Pi}_V$, it was shown in the proof of \cite[Lemma
  5]{Cesmelioglu:2020} that
  \begin{equation}
    \label{eq:tnormpivpibv}
    \tnorm{(\Pi_Vv, \bar{\Pi}_Vv)}_{v,s}
    \le C\norm[0]{v}_{1,\Omega^s},
    \quad
    \norm[0]{(\bar{\Pi}_Vv)^t}_{\Gamma^I}
    \le C\norm[0]{v}_{1,\Omega^s}.    
  \end{equation}
  (The proof in \cite[Lemma 5]{Cesmelioglu:2020} assumes that
  $\bar{\Pi}_V$ is the restriction of the Scott--Zhang interpolant to
  the mesh skeleton, but it holds also for the $L^2$-projection used
  here.) Next, by definition,
  \begin{equation*}
    \begin{split}
      \tnorm{(\Pi_V v, \bar{\Pi}_V v)}_{v,d}^2
      &=
      \norm[0]{\Pi_Vv}_{\text{div}}^2 
      + \sum_{F \in \mathcal{F}^d\backslash\mathcal{F}^I}h_F^{-1}\norm[0]{\jump{\Pi_Vv\cdot n}}_{F}^2
      + \sum_{K \in \mathcal{T}^d} h_K^{-1}\norm[0]{(\Pi_Vv - \bar{\Pi}_Vv)\cdot n}_{\partial K \cap \Gamma^I}^2
      \\
      &=: I_1 + I_2 + I_3.
    \end{split}
  \end{equation*}
  Let us consider each term separately. For $I_1$, by a triangle
  inequality and \cref{eq:PiVintest-u,eq:PiVintest-divu},
  $I_1 \le C \norm{v}_{1,\Omega^d}^2$.  Since
  $\Pi_Vv \in H(\text{div};\Omega^d)$ and $v=0$ on $\Gamma^d$ we find
  that $I_2 = 0$. Finally, for $I_3$, we have
  \begin{equation*}
    I_3
    = \sum_{K \in \mathcal{T}^d} h_K^{-1}\norm[0]{(\Pi_Vv - \bar{\Pi}_Vv)\cdot n}_{\partial K \cap \Gamma^I}^2
    \le C \sum_{K \in \mathcal{T}^d} h_K^{-1} h_K \norm{v}_{1,K}^2 = C \norm{v}_{1,\Omega^d}^2,
  \end{equation*}
  where the inequality follows from the proof of \cite[Lemma
  9]{Rhebergen:2020}. Collecting the bounds for $I_1$ to $I_3$ proves
  $\tnorm{(\Pi_V v, \bar{\Pi}_V v)}_{v,d} \le C
  \norm{v}_{1,\Omega^d}$. Combining this result with
  \cref{eq:tnormpivpibv}, we find
  \begin{equation*}
    \tnorm{(\Pi_V v, \bar{\Pi}_V v)}_{v}
    \le C \norm{v}_{1,\Omega}.
  \end{equation*}
  At this point, we remark that
  $(\Pi_V v, \bar{\Pi}_V v) \in \text{Ker}(b_h^2)$. To see this, we
  note that
  \begin{equation*}
    \begin{split}
      b_h^{2}((\Pi_V v, \bar{\Pi}_V v), (\bar{q}^s_h,\bar{q}^d_h))
      &= \int_{\Gamma^I \cup \Gamma^s} \bar{q}^s_h (\Pi_Vv - \bar{\Pi}_Vv) \cdot n^s \dif s
      + \int_{\Gamma^I \cup \Gamma^d} \bar{q}^d_h (\Pi_Vv - \bar{\Pi}_Vv) \cdot n^d \dif s
      \\
      &= \int_{\Gamma^I} \bar{q}^s_h (v - v) \cdot n^s \dif s
      + \int_{\Gamma^I} \bar{q}^d_h (v - v) \cdot n^d \dif s = 0,
    \end{split}
  \end{equation*}
  where the first equality is because $\Pi_Vv \cdot n^j$ is continuous
  on element boundaries and $\bar{q}_h^j$ is single-valued, and the
  second equality is by properties of $\Pi_V$ and $\bar{\Pi}_V$, and
  $v \cdot n^j = 0$ on $\Gamma^j$. We therefore find,
  \begin{equation*}
    \sup_{\substack{\boldsymbol{v}_h \in \text{Ker}(b_h^2) \\ \boldsymbol{v}_h \ne 0}}
    \frac{-\int_{\Omega}q_h\nabla \cdot v_h\dif x}{\tnorm{\boldsymbol{v}_h}_{v}}
    \ge
    \frac{-\int_{\Omega}q_h\nabla \cdot \Pi_Vv\dif x}{\tnorm{(\Pi_Vv,\bar{\Pi}_V v)}_{v}}
    \ge
    \frac{\norm[0]{q_h}^2_{\Omega}}{C\norm[0]{v}_{1,\Omega}}
    \ge
    \frac{c_{vq}}{C}\norm[0]{q_h}_{\Omega},
  \end{equation*}
  where we used $c_{vq}\norm{v}_{1,\Omega} \le \norm[0]{q_h}_{\Omega}$
  for the last inequality.

  With \cref{eq:infsup-b1} proven, we proceed with
  \cref{eq:infsup-b2}. Let $\bar{q}_h^j \in \bar{Q}_h^j$. Define \\
  $R_k(\partial K) := \cbr[0]{q\,: q\in L^2(\partial K),\ q|_F \in
    P_k(F)\ \forall F\in\mathcal{F}(K)}$ where $\mathcal{F}(K)$ is the
  set of facets of the simplex $K$. Let
  $w_h^j \in \sbr[0]{P_k(K)}^{\dim}$ for all $K \in \mathcal{T}^d$
  such that $w_h^j := L\bar{q}_h^j$, $j=s,d$, with
  $L:R_k(\partial K) \to \sbr[0]{P_k(K)}^{\dim}$ the
  Brezzi--Douglas--Marini (BDM) local lifting operator (for example,
  \cite[Proposition 2.10]{Du:book}). Define
  $w_h := \chi^sw_h^s + \chi^dw_h^d \in \sbr[0]{P_k(K)}^{\dim}$ for all
  $K \in \mathcal{T}$. It was shown in the proof of \cite[Lemma
  6]{Cesmelioglu:2020} that
  \begin{equation}
    \label{eq:infsup-7}
    \tnorm{(w_h,0)}_{v,s}^2
    = \sum_{K\in \mathcal{T}^s}\del[2]{\norm[0]{\nabla (L\bar{q}_h^s)}_K^2 + h_K^{-1}\norm[0]{L\bar{q}_h^s}_{\partial K}^2}
    \le C\sum_{K\in \mathcal{T}^s}h_K^{-1}\norm[0]{\bar{q}_h^s}_{\partial K}^2.
  \end{equation}
  Next, note that
  \begin{equation*}
    \tnorm{(w_h^d, 0)}_{v,d}^2
    =
    \norm[0]{w_h^d}_{\text{div}}^2       
    + \sum_{F \in \mathcal{F}^d\backslash\mathcal{F}^I}h_F^{-1}\norm[0]{\jump{w_h^d\cdot n}}_{F}^2
    + \sum_{K \in \mathcal{T}^d} h_K^{-1}\norm[0]{w_h^d\cdot n}_{\partial K \cap \Gamma^I}^2
    =: J_1 + J_2 + J_3.      
  \end{equation*}
  From \cite[Eq. (44)]{Cesmelioglu:2020}, and similar to
  \cref{eq:infsup-7},
  \begin{equation*}
    J_1 \le \sum_{K\in\mathcal{T}^d}\del[1]{\norm[0]{L\bar{q}_h^d}_K^2 + \norm[0]{\nabla (L\bar{q}_h^d)}_K^2}
    \le C\sum_{K\in\mathcal{T}^d}h_K^{-1}\norm[0]{\bar{q}_h^d}_{\partial K}^2.
  \end{equation*}
  Since $L\bar{q}_h^d \cdot n = \bar{q}_h^d$ on $\partial K$ for all
  $K \in \mathcal{T}^d$,
  \begin{equation*}
    J_2 \le C\sum_{K \in \mathcal{T}^d} h_K^{-1}\norm[0]{L\bar{q}_h^d \cdot n}_{\partial K}^2
    = C \sum_{K \in \mathcal{T}^d} h_K^{-1}\norm[0]{\bar{q}_h^d}_{\partial K}^2.
  \end{equation*}
  Finally, for $J_3$ we find
  \begin{equation*}
    J_3 = \sum_{K \in \mathcal{T}^d} h_K^{-1}\norm[0]{L\bar{q}_h^d \cdot n}_{\partial K \cap \Gamma^I}^2
    =\sum_{K \in \mathcal{T}^d} h_K^{-1}\norm[0]{\bar{q}_h^d}_{\partial K \cap \Gamma^I}^2
    \le \sum_{K\in\mathcal{T}^d}h_K^{-1}\norm[0]{\bar{q}_h^d}_{\partial K}^2.
  \end{equation*}
  Collecting the bounds for $J_1$ to $J_3$, we find
  $ \tnorm{(w_h^d, 0)}_{v,d}^2 \le C \sum_{K \in \mathcal{T}^d}
  h_K^{-1} \norm[0]{\bar{q}_h^d}_{\partial K}^2$. Combining this with
  \cref{eq:infsup-7}, 
  \begin{equation*}
    \label{eq:whvterm}
    \tnorm{(w_h, 0)}_{v}^2
    \le C \sum_{j=s,d}\del[2]{\sum_{K \in \mathcal{T}^j} h_K^{-1} \norm[0]{\bar{q}_h^j}_{\partial K}^2}.
  \end{equation*}
  Using that $w_h^j\cdot n=\bar{q}_h^j$, $j=s,d$, \cref{eq:infsup-b2}
  now follows using identical steps as the proof of \cite[Lemma
  6]{Cesmelioglu:2020}. \qed
\end{proof}

The next lemma proves boundedness of $b_h$.

\begin{lemma}
  \label{lem:boundedbh}
  There exists a positive constant $c_{bc}$, such that for all
  $(\boldsymbol{v}, \boldsymbol{q}) \in \boldsymbol{X}(h) \times
  \boldsymbol{Q}(h)$,
  \begin{equation*}
    |b_h(\boldsymbol{v}, \boldsymbol{q})|
    \le c_{bc} \tnorm{\boldsymbol{v}}_v \tnorm{\boldsymbol{q}}_p.
  \end{equation*}
\end{lemma}
\begin{proof}
  Note that
  \begin{equation*}
    b_h(\boldsymbol{v}, \boldsymbol{q})
    = \underbrace{b_h^s(v, \boldsymbol{q}^s) + b_h^{I,s}(\bar{v}, \bar{q}^s)}_{I_1}
    + \underbrace{b_h^d(v, \boldsymbol{q}^d) + b_h^{I,d}(\bar{v}, \bar{q}^d)}_{I_2}.
  \end{equation*}
  Let us consider $I_1$ and $I_2$ separately. For $I_1$, we have 
  \begin{equation*}
    \begin{split}
      \envert{I_1}
      =&
       \envert[2]{  -\sum_{K\in\mathcal{T}^s} \int_K q \nabla\cdot v \dif x
      + \sum_{K\in\mathcal{T}^s} \int_{\partial K} \bar{q}^s (v - \bar{v}) \cdot n^s \dif s }
      \\
      \le&
      \del[1]{\sum_{K \in \mathcal{T}^s} \norm[0]{\nabla v}_{K}^2 + \sum_{K \in \mathcal{T}^s} h_K^{-1}\norm[0]{v - \bar{v}}_{\partial K}^2 }^{1/2}
      \del[1]{\, \norm{q}_{\Omega^s}^2 + \sum_{K \in \mathcal{T}^s} h_K\norm[0]{\bar{q}^s}_{\partial K}^2 }^{1/2}
      \\
      \le& \tnorm{\boldsymbol{v}}_v \tnorm{\boldsymbol{q}}_p.     
    \end{split}
  \end{equation*}
  Next, consider $I_2$. We have
  \begin{equation*}
    \begin{split}
      |I_2|
      =&
      \envert[2]{-\sum_{K\in\mathcal{T}^d} \int_K q \nabla\cdot v \dif x
        + \sum_{F \in \mathcal{F}^d\backslash\mathcal{F}^I} \int_F \bar{q}^d \jump{v \cdot n} \dif s
        + \sum_{F \in \mathcal{F}^I} \int_F \bar{q}^d(v^d - \bar{v}) \cdot  n^d \dif s}
      \\      
      \le&
      \del[2]{\sum_{K\in\mathcal{T}^d}\norm[0]{\nabla \cdot v}_{K}^2
        + \sum_{F \in \mathcal{F}^d\backslash\mathcal{F}^I}h_F^{-1}\norm[0]{\jump{v\cdot n}}_F^2
        + \sum_{F \in \mathcal{F}^I}h_F^{-1}\norm[0]{(v^d - \bar{v})\cdot n}_F^2 }^{1/2}
      \\
      & \times \del[2]{\norm{q}^2_{\Omega^d} + \sum_{K \in \mathcal{T}^d}h_K\norm[0]{\bar{q}^d}_{\partial K}^2
        + \sum_{K \in \mathcal{T}^d}h_K\norm[0]{\bar{q}^d}_{\partial K}^2}^{1/2}
      \\
      \le& \sqrt{2}\tnorm{\boldsymbol{v}}_v \tnorm{\boldsymbol{q}}_p.      
    \end{split}
  \end{equation*}
  The result follows by combining the bounds for $|I_1|$ and $|I_2|$. \qed
\end{proof}

The form $t_h$ \cref{eq:oh} is new compared to
\cite{Cesmelioglu:2020}. We have the following properties of this
term.

\begin{lemma}
  \label{lem:boundedness_th}
  Let $w_1,w_2 \in X^s(h)$ such that $\nabla \cdot w_j = 0$ on each
  $K \in \mathcal{T}^s$ and $w_j \in H({\rm div};\Omega^s)$, for
  $j=1,2$. Then for any
  $\boldsymbol{u},\boldsymbol{v} \in \boldsymbol{X}(h)$, there exists
  a constant $c_w>0$ such that:
  \begin{equation*}
    \envert[0]{t_h(w_1; \boldsymbol{u}, \boldsymbol{v}) - t_h(w_2; \boldsymbol{u}, \boldsymbol{v})}
    \le c_w \norm{w_1 - w_2}_{1,h,\Omega^s} \tnorm{\boldsymbol{u}}_{v,s} \tnorm{\boldsymbol{v}}_{v,s}.
  \end{equation*}
\end{lemma}
\begin{proof}
  After integrating $t_h(w_j; \boldsymbol{u}, \boldsymbol{v})$ by parts and
  using that
  \begin{equation*}
    \int_{\Gamma^I} (w_j\cdot n)\bar{u}\cdot\bar{v} \dif s = \sum_{K \in
      \mathcal{T}^s} \int_{\partial K} (w_j \cdot n)\bar{u}\cdot\bar{v}
    \dif s, \quad j=1,2,
  \end{equation*}
  which holds since $w_j \cdot n$ is continuous, $\bar{u}$ and
  $\bar{v}$ are single-valued on element boundaries, and
  $\bar{u}=\bar{v}=0$ on $\Gamma^s$, the remainder of the proof is
  identical to that of the proof of \cite[Proposition
  3.4]{Cesmelioglu:2017}. \qed
\end{proof}

Next, we remark that for $w \in X^s(h) \cap H({\rm div};\Omega^s)$
such that $\nabla \cdot w = 0$ on each $K \in \mathcal{T}^s$, and
$\boldsymbol{v} \in \boldsymbol{X}(h)$ that
\begin{equation}
  \label{eq:coer_th}
  t_h(w; \boldsymbol{v}, \boldsymbol{v}) = \tfrac{1}{2} \sum_{K \in \mathcal{T}^s} \int_{\partial K} |w\cdot n||v-\bar{v}|^2 \dif s
  + \tfrac{1}{2}\int_{\Gamma^I} (w\cdot n) |\bar{v}|^2 \dif s.
\end{equation}

Combining some of the above results, we have the following two lemmas.

\begin{lemma}
  \label{lem:boundedness_ahwuv}
  Let $w \in X^s(h) \cap H({\rm div};\Omega^s)$ such that
  $\nabla \cdot w = 0$ on each $K \in \mathcal{T}^s$. Furthermore, let
  $\boldsymbol{u},\boldsymbol{v} \in \boldsymbol{X}(h)$ and
  $\boldsymbol{u}_h,\boldsymbol{v}_h \in \boldsymbol{X}_h$. Then
  \begin{equation*}
    \envert[0]{a_h(w;\boldsymbol{u},\boldsymbol{v})}
    \le c_{ac} \mu \tnorm{\boldsymbol{u}}_{v'}\tnorm{\boldsymbol{v}}_{v'},
  \end{equation*}
  with
  $c_{ac} = 2c_e^2\max(c_w\mu^{-1}\norm{w}_{1,h,\Omega^s} + c_{ac}^s,
  \kappa_{\min}^{-1}, \alpha \kappa_{\min}^{-1/2})$.
\end{lemma}
\begin{proof}
  Note that by \cref{eq:continuity-ah-sep-terms} and
  \cref{lem:boundedness_th} with $w_2=0$, 
  \begin{align*}
    \envert[0]{a_h(w;\boldsymbol{u},\boldsymbol{v})}
    \le& \envert[0]{t_h(w; \boldsymbol{u}, \boldsymbol{v})}
        + \envert[0]{a_h^s(\boldsymbol{u}, \boldsymbol{v})}
         + \envert[0]{a^d(u, v)} + \envert[0]{a^I(\bar{u}, \bar{v})}
    \\
    \le& c_w\norm{w}_{1,h,\Omega^s}\tnorm{\boldsymbol{u}}_{v,s} \tnorm{\boldsymbol{v}}_{v,s}
         + \mu c_{ac}^s \tnorm{\boldsymbol{u}}_{v',s} \tnorm{\boldsymbol{v}}_{v',s}
         \\
        &+ \mu \kappa_{\min}^{-1}\norm{u}_{\Omega^d}\norm{v}_{\Omega^d}
         + \alpha\mu\kappa_{\min}^{-1/2}\norm[0]{\bar{u}^t}_{\Gamma^I}\norm[0]{\bar{v}^t}_{\Gamma^I}
    \\
    \le& 2\max(c_w\norm{w}_{1,h,\Omega^s} + \mu c_{ac}^s, \mu\kappa_{\min}^{-1}, \alpha\mu\kappa_{\min}^{-1/2})
         \\
         & \times\del[1]{\tnorm{\boldsymbol{u}}_{v',s}^2 + \norm{u}_{\Omega^d}^2 + \norm[0]{\bar{u}^t}_{\Gamma^I}^2}^{1/2}
           \del[1]{\tnorm{\boldsymbol{v}}_{v',s}^2 + \norm{v}_{\Omega^d}^2 + \norm[0]{\bar{v}^t}_{\Gamma^I}^2}^{1/2}.
  \end{align*}
  The results follows by definition of $\tnorm{\cdot}_{v'}$ and using
  that the norm equivalency constant $c_e \ge 1$. \qed
\end{proof}

\begin{remark}
  \Cref{lem:boundedness_ahwuv} holds also for
  $\boldsymbol{u} \in \boldsymbol{X}_h$ with
  $\tnorm{\boldsymbol{u}}_{v'}$ replaced by
  $\tnorm{\boldsymbol{u}}_{v}$ and/or
  $\boldsymbol{v} \in \boldsymbol{X}_h$ with
  $\tnorm{\boldsymbol{v}}_{v'}$ replaced by
  $\tnorm{\boldsymbol{v}}_{v}$ due to the equivalence of the norms
  $\tnorm{\cdot}_{v}$ and $\tnorm{\cdot}_{v'}$ on $\boldsymbol{X}_h$.
\end{remark}

Using a similar approach as in \cite[Lemma 2]{Discacciati:2017}, we
show the coercivity of $a_h$.

\begin{lemma}
  \label{lem:coercivity_awhvhvh}
  Let $w \in X^s(h) \cap H({\rm div};\Omega^s)$ such that
  $\nabla \cdot w = 0$ on each $K \in \mathcal{T}^s$, and
  $\norm{w \cdot n}_{\Gamma^I} \le \mu c_{ae}^s \delta
  (c_{pq}^2+c_{si,4}^2)^{-1}$ with $0 < \delta < 1$. Then, for
  $\beta > \beta_0$,
  \begin{equation}
    \label{eq:coercivity_awhvhvh}
    a_h(w; \boldsymbol{v}_h, \boldsymbol{v}_h)
    \ge c_{ae} \mu
    \tnorm{\boldsymbol{v}_h}_{v}^2
    \quad \forall \boldsymbol{v}_h \in \boldsymbol{Z}_h,
  \end{equation}
  where
  $c_{ae} = \min\del[1]{(1-\delta) c_{ae}^s, \kappa_{\max}^{-1},
    \alpha\kappa_{\max}^{-1/2}} > 0$.
\end{lemma}
\begin{proof}
  Let $\boldsymbol{v}_h \in \boldsymbol{Z}_h$. From \cref{eq:coer_th}
  we note that
  \begin{equation}
    \label{eq:thgeterms}
    t_h(w; \boldsymbol{v}_h, \boldsymbol{v}_h)
    \ge
    - \tfrac{1}{2}\int_{\Gamma^I} \envert{w \cdot n} |\bar{v}_h|^2 \dif s
    \ge
    - \int_{\Gamma^I} \envert{w \cdot n} |\bar{v}_h-v_h^s|^2 \dif s - \int_{\Gamma^I} \envert{w \cdot n} |v_h^s|^2 \dif s,
  \end{equation}
  where the second inequality is due to
  $|\bar{v}_h|^2 \le 2|\bar{v}_h - v_h^s|^2 + 2|v_h^s|^2$. Let us
  consider each term on the right hand side of \cref{eq:thgeterms}
  separately. For the first term we find, using the Cauchy--Schwarz
  inequality,
  \begin{equation*}
    \int_{\Gamma^I} \envert{w \cdot n} |\bar{v}_h-v_h^s|^2 \dif s
    \le \norm{w\cdot n}_{\Gamma^I} \norm[0]{\bar{v}_h-v_h^s}_{4,0,\Gamma^I}^2.
  \end{equation*}
  By a scaling identity, for $\mu \in R_k(\partial K)$, we have that
  there exists a positive constant $c_{pq}$ independent of $h$ such
  that
  $\norm[0]{\mu}_{4,0,\partial K} \le c_{pq}
  h^{(1-d)/4}\norm[0]{\mu}_{\partial K}$. We therefore find:
  \begin{equation}
    \label{eq:thgeterms_t1}
    \int_{\Gamma^I} \envert{w \cdot n} |\bar{v}_h-v_h^s|^2 \dif s
    \le c_{pq}^2\norm{w\cdot n}_{\Gamma^I} h^{(1-d)/2}\norm[0]{\bar{v}_h-v_h^s}_{\Gamma^I}^2
    \le c_{pq}^2\norm{w\cdot n}_{\Gamma^I} \tnorm{\boldsymbol{v}_h}_{v,s}^2,
  \end{equation}
  where the second inequality is true for $d=2,3$. For the second term on
  the right hand side of \cref{eq:thgeterms}, using the
  Cauchy--Schwarz inequality and \cref{eq:dtrpoincareineq} with $r=4$,
  we obtain:
  \begin{equation}
    \label{eq:thgeterms_t2}
    \int_{\Gamma^I} \envert{w \cdot n} |v_h^s|^2 \dif s
    \le \norm{w\cdot n}_{\Gamma_I} \norm{v_h^s}^2_{4,0,\Gamma^I}
    \le c_{si,4}^2\norm{w\cdot n}_{\Gamma_I}\tnorm{\boldsymbol{v}_h}_{v,s}^2.
  \end{equation}
  Combining \cref{eq:thgeterms,eq:thgeterms_t1,eq:thgeterms_t2}, we
  find
  $ t_h(w; \boldsymbol{v}_h, \boldsymbol{v}_h) \ge
  -(c_{pq}^2+c_{si,4}^2)\norm{w\cdot
    n}_{\Gamma_I}\tnorm{\boldsymbol{v}_h}_{v,s}^2$. Using this
  inequality together with \cref{eq:coercivity_ahsdI},
  \begin{equation*}
    \begin{split}
      a_h(w; \boldsymbol{v}_h, \boldsymbol{v}_h)
      \ge &
      \mu c_{ae}^s \tnorm{\boldsymbol{v}_h}_{v,s}^2
      + \mu \kappa_{\max}^{-1} \norm{v_h}_{\Omega^d}^2
      + \alpha\mu\kappa_{\max}^{-1/2}\norm[0]{\bar{v}_h^t}_{\Gamma^I}^2
       -(c_{pq}^2+c_{si,4}^2)\norm{w\cdot n}_{\Gamma_I}\tnorm{\boldsymbol{v}_h}_{v,s}^2
      \\
      \ge &
      \del[1]{\mu c_{ae}^s - (c_{pq}^2+c_{si,4}^2)\norm{w\cdot n}_{\Gamma_I}}\tnorm{\boldsymbol{v}_h}_{v,s}^2
      + \mu \kappa_{\max}^{-1} \norm{v_h}_{\Omega^d}^2
      + \alpha\mu\kappa_{\max}^{-1/2}\norm[0]{\bar{v}_h^t}_{\Gamma^I}^2
      \\
      \ge &
      (1-\delta)\mu c_{ae}^s\tnorm{\boldsymbol{v}_h}_{v,s}^2
      + \mu \kappa_{\max}^{-1} \norm{v_h}_{\Omega^d}^2
      + \alpha\mu\kappa_{\max}^{-1/2}\norm[0]{\bar{v}_h^t}_{\Gamma^I}^2,
    \end{split}
  \end{equation*}
  where the last step is by the assumption on
  $\norm{w \cdot n}_{\Gamma^I}$. The result follows by definition of
  $\tnorm{\boldsymbol{v}_h}_v$ noting that
  $\tnorm{\boldsymbol{v}_h}_{v,d} = \norm{v_h}_{\Omega^d}^2$ for
  $\boldsymbol{v}_h \in \boldsymbol{Z}_h$. \qed
\end{proof}

Let us define the following space:
\begin{equation*}
  \label{eq:def_Bhs}
  \boldsymbol{B}_h^s := \cbr[1]{ \boldsymbol{v}_h \in \boldsymbol{Z}_h^s:\
    \tnorm{\boldsymbol{v}_h}_{v,s} \le c_{ae}^{-1}\del[1]{\mu^{-1}c_p\norm[0]{f^s}_{\Omega^s} + 2c_fc_{bb}^{-1} \norm[0]{f^d}_{\Omega^d}}},
\end{equation*}
where
$c_f =
2c_e^2\max\del[0]{c_wc_{ae}^s\delta/(c_{si,2}(c_{pq}^2+c_{si,4}^2)) +
  c_{ac}^s, \kappa_{\min}^{-1},\alpha\kappa_{\min}^{-1/2}}$.

Similar to \cite{Caucao:2020}, let us also define the fixed point
operator
$\boldsymbol{\Psi}_h : \boldsymbol{B}_h^s \to \boldsymbol{B}_h^s$ by
$\boldsymbol{\Psi}_h(\boldsymbol{w}_h^s):=\boldsymbol{u}_h^s$ for all
$\boldsymbol{w}_h^s \in \boldsymbol{B}_h^s$ where $\boldsymbol{u}_h^s$
is the restriction of $\boldsymbol{u}_h$ to $\Omega^s$ and
$\boldsymbol{u}_h$ is the solution to the linear problem:
Given $w_h \in X_h^s$ that satisfies all the conditions in
\cref{lem:coercivity_awhvhvh},
$f^s \in \sbr[0]{L^2(\Omega^s)}^{\dim}$, and $f^d \in L^2(\Omega^d)$,
find
$(\boldsymbol{u}_h, \boldsymbol{p}_h) \in \boldsymbol{X}_h \times
\boldsymbol{Q}_h$ such that
\begin{equation*}
  a_h(w_h; \boldsymbol{u}_h, \boldsymbol{v}_h) + b_h(\boldsymbol{v}_h, \boldsymbol{p}_h)
  + b_h(\boldsymbol{u}_h, \boldsymbol{q}_h)
  = \ell^s(v_h) + \ell^d(q_h),
  \quad \forall (\boldsymbol{v}_h, \boldsymbol{q}_h) \in \boldsymbol{X}_h \times \boldsymbol{Q}_h.
\end{equation*}  
Then
$(\boldsymbol{u}_h, \boldsymbol{p}_h) \in \boldsymbol{X}_h \times
\boldsymbol{Q}_h$ is the solution to the Navier--Stokes/Darcy problem
\cref{eq:hdg} if and only if
$\boldsymbol{\Psi}_h(\boldsymbol{u}_h^s)=\boldsymbol{u}_h^s$.

Using Brouwer's and Banach's fixed point theorems, well-posedness of
the Navier--Stokes/Darcy problem \cref{eq:hdg} now follows using
Lemmas \ref{lem:consistency} to \ref{lem:coercivity_awhvhvh} by
showing that $\boldsymbol{\Psi}_h$ has a unique fixed point. See
\cite{Discacciati:2017} for details. In particular, the following
result holds:

\begin{theorem}
  Let $f^s \in \sbr[0]{L^2(\Omega^s)}^{\dim}$ and
  $f^d \in L^2(\Omega^d)$ satisfy assumption
  \begin{equation}
    \label{eq:assumption_fs_fd}
    c_p\norm[0]{f^s}_{\Omega^s} + 2\mu c_fc_{bb}^{-1} \norm[0]{f^d}_{\Omega^d}
    \le \mu^2 c_{ae}c_{ae}^s\delta c_{si,2}^{-1}(c_{pq}^2 + c_{si,4}^2)^{-1}.
  \end{equation}
  Then there exists a solution
  $(\boldsymbol{u}_h, \boldsymbol{p}_h) \in \boldsymbol{X}_h \times
  \boldsymbol{Q}_h$ to \cref{eq:hdg} that satisfies
  \begin{subequations}
    \label{eq:existunique-hdg}
    \begin{align}
      \label{eq:existunique-hdg-a}
      \tnorm{\boldsymbol{u}_h}_v
      &\le c_{ae}^{-1}\del[1]{ \mu^{-1} c_p\norm[0]{f^s}_{\Omega^s} + 2c_fc_{bb}^{-1} \norm[0]{f^d}_{\Omega^d} },
      \\
      \label{eq:existunique-hdg-b}
      \tnorm{\boldsymbol{p}_h}_p
      &\le 2c_fc_{bb}^{-1} c_{ae}^{-1}\del[1]{ c_p \norm[0]{f^s}_{\Omega^s} + \mu c_fc_{bb}^{-1} \norm[0]{f^d}_{\Omega^d} }.
    \end{align}
  \end{subequations}
  This solution is unique if the data also satisfy
  \begin{equation}
    \label{eq:data_unique_sol_uh}
    c_p\norm[0]{f^s}_{\Omega^s} + 2\mu c_fc_{bb}^{-1} \norm[0]{f^d}_{\Omega^d}
    < \mu^2c_{ae} \min\del[1]{c_{ae} c_w^{-1}, c_{ae}^s\delta c_{si,2}^{-1}(c_{pq}^2 + c_{si,4}^2)^{-1}}.
  \end{equation}
\end{theorem}

\section{Error analysis}
\label{s:errorAnalysis}

Besides the BDM interpolation operator
$\Pi_V:\sbr[0]{H^1(\Omega)}^{\dim} \to X_h \cap H(\text{div};\Omega)$
satisfying \cref{eq:PiVinterpOp,eq:PiVintest}, and the
$L^2$-projection
$\bar{\Pi}_V:\sbr[0]{H^1(\Omega^s)}^{\dim} \to \bar{X}_h$ used
previously in \cref{ss:prelim}, let $\Pi_Q$ and $\bar{\Pi}_Q^j$ be the
$L^2$-projection operators onto, respectively, $Q_h$, and
$\bar{Q}_h^j$, $j=s,d$. The interpolation and approximation errors are
defined as:
\begin{equation}
  \label{eq:inter_approx_errors}
\begin{aligned}
  e_u^I &= u - \Pi_Vu, & e_p^I &= p - \Pi_Qp, & \bar{e}_u^I &= \gamma(u)-\bar{\Pi}_Vu, & \bar{e}_{p,j}^I &= \gamma(p) - \bar{\Pi}_Q^jp,
  \\
  e_u^h &= u_h - \Pi_Vu, & e_p^h &= p_h - \Pi_Qp, & \bar{e}_u^h &= \bar{u}_h-\bar{\Pi}_Vu, & \bar{e}_{p,j}^h &= \bar{p}_h^j - \bar{\Pi}_Q^jp.
\end{aligned}  
\end{equation}
Note that $u - u_h = e_u^I - e_u^h$ and likewise for the other
unknowns. Similar to the notation used in previous sections, we write
$\boldsymbol{e}_u^r = (e_u^r, \bar{e}_u^r)$,
$\boldsymbol{e}_{u,s}^r = (e_{u,s}^r,\bar{e}_u^r)$,
$\boldsymbol{e}_{p,j}^r = (e_p^r, \bar{e}_{p,j}^r)$,
$\boldsymbol{e}_p^r = (e_p^r, \bar{e}_{p,s}^r, \bar{e}_{p,d}^r)$, for
$j=s,d$, $r=I,h$, and we remark that $e_{u,s}^r$ is the restriction of
$e_u^r$ to $\Omega^s$. From \cite[Lemma 8]{Cesmelioglu:2020} we have
that for $p^j \in H^l(\Omega^j)$, $0 \le l \le k$, $j=s,d$, that
\begin{equation}
  \label{eq:interpestimate_p}
  \tnorm{\boldsymbol{e}_p^I}_p \le C h^l\norm[0]{p}_{l,\Omega}.
\end{equation}
The following lemma, which is a modification of \cite[Lemma
7]{Cesmelioglu:2020}, determines the interpolation estimate for the
velocity field.

\begin{lemma}
  Suppose that $u \in \sbr[0]{H^l(\Omega)}^{\dim}$ for
  $2 \le l \le k+1$. Then
  \begin{equation}
    \label{eq:interpestimate_u}
    \tnorm{\boldsymbol{e}_u^I}_{v'} \le C h^{l-1}\norm{u}_{l,\Omega}.
  \end{equation}
\end{lemma}
\begin{proof}
  It was shown in the proof of \cite[Lemma 7]{Cesmelioglu:2020} that
  \begin{equation*}
    \label{eq:interpestimate_us}
    \del[0]{\tnorm{\boldsymbol{e}_u^I}_{v',s}^2 + \norm[0]{\bar{e}_u^t}_{\Gamma^I}^2}^{1/2} \le C h^{l-1}\norm{u}_{l,\Omega^s}.
  \end{equation*}
  To complete the proof, we observe by definition, using
  \cref{eq:PiVintest}, and the proof of \cite[Lemma
  9]{Rhebergen:2020}, that
  \begin{equation*}
    \tnorm{\boldsymbol{e}_u^I}_{v,d}^2
    =
    \norm[0]{e_u^I}_{\text{div}}^2
    + \sum_{K \in \mathcal{T}^d} h_K^{-1}\norm[0]{(e_u^I - \bar{e}_u^I)\cdot n}_{\partial K \cap \Gamma^I}^2
    \le ch^{2l-2}\norm[0]{u}_{l,\Omega}^2.
  \end{equation*}
  \qed
\end{proof}

To obtain the error equation, note that because $\Pi_Q$ and
$\bar{\Pi}_Q^j$, $j=s,d$ are the $L^2$-projection operators onto,
respectively $Q_h$ and $\bar{Q}_h^j$,
$\nabla \cdot v_h|_K \in P_{k-1}(K)$ for all $K \in \mathcal{T}$,
$v_h\cdot n^j|_F \in P_k(F)$ for all $F \in \mathcal{F}^j$ and
$\bar{v}_h \cdot n^j|_F \in P_k(F)$ for all $F \in \mathcal{F}^I$,
\begin{equation}
  \label{eq:bhvhepI}
  b_h(\boldsymbol{v}_h, \boldsymbol{e}_p^I)
  = 0 \quad \forall \boldsymbol{v}_h \in \boldsymbol{X}_h.  
\end{equation}
Furthermore, by properties \cref{eq:PiVinterpOp} of the BDM
interpolation operator and the $L^2$-projection $\bar{\Pi}_V$, it
follows that
\begin{equation}
  \label{eq:bheuIqh}
  b_h(\boldsymbol{e}_u^I, \boldsymbol{q}_h)
  = 0 \quad \forall \boldsymbol{q}_h \in \boldsymbol{Q}_h.  
\end{equation}
Let us denote by $a_h^L$ the linear part of $a_h$, i.e., 
\begin{equation*}
  a_h^L(\boldsymbol{u}, \boldsymbol{v})
  := a_h^s(\boldsymbol{u}, \boldsymbol{v})
  + a^d(u, v) + a^I(\bar{u}, \bar{v}).
\end{equation*}
Subtracting the consistency equations (see \cref{lem:consistency})
from \cref{eq:hdg}, using
\cref{eq:bhvhepI,eq:bheuIqh,eq:inter_approx_errors} and rearranging,
we obtain the following error equation:
\begin{equation}
  \label{eq:errorequation}
  \begin{split}
    t_h(u;\boldsymbol{e}_u^h, \boldsymbol{v}_h)
    + a_h^L(\boldsymbol{e}_u^h, \boldsymbol{v}_h)
    =&
    a_h^L(\boldsymbol{e}_u^I, \boldsymbol{v}_h)
     - b_h(\boldsymbol{v}_h, \boldsymbol{e}_p^h)
    - b_h(\boldsymbol{e}_u^h, \boldsymbol{q}_h)
    \\
    &+ t_h(u;\boldsymbol{e}_u^I, \boldsymbol{v}_h)    
    + t_h(u;\boldsymbol{u}_h, \boldsymbol{v}_h) 
    - t_h(u_h;\boldsymbol{u}_h, \boldsymbol{v}_h).    
  \end{split}
\end{equation}

The error estimates for the HDG method for the Navier--Stokes
equations in \cite{Kirk:2019} are now extended here to the coupled
Navier--Stokes and Darcy problem along the same lines as
\cite{Discacciati:2017}.

\begin{theorem}[Energy estimate velocity and pressure error]
  \label{thm:energyestimate}
  Let $(u,p) \in \sbr[0]{H^{k+1}(\Omega)}^{\dim} \times H^k(\Omega)$
  be the solution to the Navier--Stokes/Darcy problem
  \cref{eq:system,eq:interface,eq:bcs_simple} and let
  $\boldsymbol{u} = (u, \gamma(u))$ and
  $\boldsymbol{p} = (p, \gamma(p^s), \gamma(p^d))$. Let
  $(\boldsymbol{u}_h,\boldsymbol{p}_h) \in \boldsymbol{X}_h \times
  \boldsymbol{Q}_h$ be the solution to the discrete
  Navier--Stokes/Darcy problem \cref{eq:hdg}. Let $C_w$, $C_p$, $C_f$,
  $C_{bb}$, $C_{ae}$, $C_{ae}^s$, $C_{si,2}$, and $C_{si,4}$ be the
  constants in \cref{eq:uniquenessrequirement-wf} and let $c_w$,
  $c_p$, $c_f$, $c_{bb}$, $c_{ae}$, $c_{ae}^s$, $c_{si,2}$,
  $c_{si,4}$, and $c_{pq}$ be the constants in
  \cref{eq:data_unique_sol_uh}. Let $\tilde{c}_w = \max(C_w, c_w)$,
  $\tilde{c}_p = \max(C_p, c_p)$, $\tilde{c}_f = \max(C_f, c_f)$,
  $\tilde{c}_{bb} = \min(C_{bb},c_{bb})$,
  $\tilde{c}_{ae} = \min(C_{ae},c_{ae})$,
  $\tilde{c}_{ae}^s = \min(C_{ae}^s,c_{ae}^s)$,
  $\tilde{c}_{sir} = \max(\tfrac{1}{2}C_{si,2}C_{si,4}^2,
  c_{si,2}(c_{pq}^2 + c_{si,4}^2))$, and let $0 <\delta < 1$. If
  \begin{equation}
    \label{eq:databoundfsfd}
    \tilde{c}_p\norm[0]{f^s}_{\Omega^s} + 2\mu \tilde{c}_f\tilde{c}_{bb}^{-1} \norm[0]{f^d}_{\Omega^d}
    < \tfrac{1}{2}\mu^2\tilde{c}_{ae} \min\del[1]{\tilde{c}_{ae}\tilde{c}_w^{-1}, \tilde{c}_{ae}^s\delta \tilde{c}_{sir}^{-1}},
  \end{equation}
  then
  \begin{subequations}
    \begin{align}
      \label{eq:apriori-u}
      \tnorm{\boldsymbol{u}-\boldsymbol{u}_h}_v
      &\le c_1 h^{k}\norm{u}_{k+1,\Omega},
      \\
      \label{eq:apriori-p}
      \tnorm{\boldsymbol{p}-\boldsymbol{p}_h}_p
      &\le c_2 h^k \del[1]{\, \norm[0]{p}_{k,\Omega} + \mu\norm{u}_{k+1,\Omega}},
    \end{align}
  \end{subequations}
  where $c_1,c_2>0$ are constants independent of $\mu$ and $h$.
\end{theorem}
\begin{proof}
  We first prove \cref{eq:apriori-u}. Take
  $(\boldsymbol{v}_h, \boldsymbol{q}_h) = (\boldsymbol{e}_u^h,
  -\boldsymbol{e}_p^h)$ in \cref{eq:errorequation}. By coercivity of
  $a_h$ \cref{lem:coercivity_awhvhvh}, we find
  \begin{equation*}
    \begin{split}
      c_{ae} \mu \tnorm{\boldsymbol{e}_u^h}_{v}^2
      \le
      a_h(u; \boldsymbol{e}_u^h, \boldsymbol{e}_u^h)
      =&
      a_h^L(\boldsymbol{e}_u^I, \boldsymbol{e}_u^h)
      + t_h(u;\boldsymbol{e}_u^I, \boldsymbol{e}_u^h)    
      + t_h(u;\boldsymbol{u}_h, \boldsymbol{e}_u^h) 
      - t_h(u_h;\boldsymbol{u}_h, \boldsymbol{e}_u^h)
      \\
      =&
      a_h(u;\boldsymbol{e}_u^I, \boldsymbol{e}_u^h)
      + \sbr[1]{t_h(u;\boldsymbol{u}_h, \boldsymbol{e}_u^h) 
        - t_h(u_h;\boldsymbol{u}_h, \boldsymbol{e}_u^h)}      
      =: I_1 + I_2.
    \end{split}
  \end{equation*}
  We bound each term separately, starting with $I_1$. Since
  \cref{eq:databoundfsfd} holds, it follows by
  \cref{eq:bounduXpQ-wf-a} that
  \begin{equation*}
      \norm{u}_{1,h,\Omega^s}
      \le \tilde{c}_{ae}^{-1}\mu^{-1}\del[1]{ \tilde{c}_p\norm[0]{f^s}_{\Omega^s}
        + 2 \tilde{c}_{f}\mu\tilde{c}_{bb}^{-1}\norm[0]{f^d}_{\Omega^d} }
      \le
      \tfrac{1}{2}\mu \min\del{\tilde{c}_{ae}\tilde{c}_w^{-1}, \tilde{c}_{ae}^s\delta \tilde{c}_{sir}^{-1}}
      \le
      \mu \tilde{c}_{ae}^s\delta \tilde{c}_{sir}^{-1}.
  \end{equation*}
  Therefore, $c_{ac}$ in \cref{lem:boundedness_ahwuv} is bounded by
  \begin{equation*}
    c_{ac}
    =
    2c_e^2\max(c_w\mu^{-1}\norm{u}_{1,h,\Omega^s} + c_{ac}^s, \kappa_{\min}^{-1}, \alpha \kappa_{\min}^{-1/2})
    \le
    2c_e^2\max(c_w\tilde{c}_{ae}^s\delta \tilde{c}_{sir}^{-1} + c_{ac}^s, \kappa_{\min}^{-1}, \alpha \kappa_{\min}^{-1/2})
    \le c_f,
  \end{equation*}
  so that, by \cref{lem:boundedness_ahwuv},
  $I_1 \le \tilde{c}_{f} \mu
  \tnorm{\boldsymbol{e}_u^I}_{v'}\tnorm{\boldsymbol{e}_u^h}_{v}$. To
  bound $I_2$, we use \cref{lem:boundedness_th} and
  \cref{eq:dpoincareineq}:
  \begin{equation*}
    \envert[0]{t_h(u; \boldsymbol{u}_h, \boldsymbol{e}_u^h) - t_h(u_h; \boldsymbol{u}_h, \boldsymbol{e}_u^h)}
    \le c_w \norm[0]{u - u_h}_{1,h,\Omega^s} \tnorm{\boldsymbol{u}_h}_{v,s} \tnorm{\boldsymbol{e}_u^h}_{v,s}
    \le c_w\del[1]{\tnorm{\boldsymbol{e}_u^I}_{v,s} + \tnorm{\boldsymbol{e}_u^h}_{v,s}}\tnorm{\boldsymbol{u}_h}_{v,s} \tnorm{\boldsymbol{e}_u^h}_{v}.
  \end{equation*}
  Combining the bounds for $I_1$ and $I_2$,
  \begin{equation}
    \label{eq:I1I2I3ehu}
      c_{ae} \mu \tnorm{\boldsymbol{e}_u^h}_{v}^2
      \le
      \tilde{c}_{f} \mu \tnorm{\boldsymbol{e}_u^h}_{v} \tnorm{\boldsymbol{e}_u^I}_{v'}
      + c_w\tnorm{\boldsymbol{e}_u^I}_{v,s}\tnorm{\boldsymbol{u}_h}_{v,s} \tnorm{\boldsymbol{e}_u^h}_{v}
      + c_w\tnorm{\boldsymbol{u}_h}_{v,s} \tnorm{\boldsymbol{e}_u^h}_{v}^2.    
  \end{equation}
  By \cref{eq:existunique-hdg-a} and \cref{eq:databoundfsfd},
  \begin{equation}
    \label{eq:cwuhupbound}
    \begin{split}
      c_w\tnorm{\boldsymbol{u}_h}_{v,s}
      \le \tilde{c}_w\tnorm{\boldsymbol{u}_h}_v
      &\le \tilde{c}_{ae}^{-1}\mu^{-1}\tilde{c}_w\del[1]{ \tilde{c}_p\norm[0]{f^s}_{\Omega^s} + 2\tilde{c}_f\mu\tilde{c}_{bb}^{-1} \norm[0]{f^d}_{\Omega^d}}
      \\
      &\le \tfrac{1}{2}\tilde{c}_w
      \mu \min\del{ \tilde{c}_{ae} \tilde{c}_w^{-1}, \tilde{c}_{ae}^s\delta \tilde{c}_{sir}^{-1}}
      \le \tfrac{1}{2}\tilde{c}_{ae}\mu.
    \end{split}
  \end{equation}
  Combining \cref{eq:I1I2I3ehu,eq:cwuhupbound},
  \begin{equation*}
    \label{eq:erreuhv2bound}
      \tfrac{1}{2}\tilde{c}_{ae} \mu \tnorm{\boldsymbol{e}_u^h}_{v}^2
      \le
      (c_{ae}-\tfrac{1}{2}\tilde{c}_{ae}) \mu \tnorm{\boldsymbol{e}_u^h}_{v}^2
      \le
      \tilde{c}_{f} \mu \tnorm{\boldsymbol{e}_u^h}_{v} \tnorm{\boldsymbol{e}_u^I}_{v'}
      + \tfrac{1}{2}\tilde{c}_{ae}\mu\tnorm{\boldsymbol{e}_u^I}_{v,s}\tnorm{\boldsymbol{e}_u^h}_{v},
  \end{equation*}
  resulting in
  \begin{equation}
    \label{eq:approxininterpeu}
    \tnorm{\boldsymbol{e}_u^h}_{v}
    \le      
    2\tilde{c}_{f}\tilde{c}_{ae}^{-1} \tnorm{\boldsymbol{e}_u^I}_{v'}
    + \tnorm{\boldsymbol{e}_u^I}_{v,s}
    \le (1+2\tilde{c}_{f}\tilde{c}_{ae}^{-1})\tnorm{\boldsymbol{e}_u^I}_{v'}.
  \end{equation}
  Applying a triangle inequality to
  $\tnorm{\boldsymbol{u}-\boldsymbol{u}_h}_v$ and using
  \cref{eq:approxininterpeu} results in
  \begin{equation}
    \label{eq:uuheuI}
    \tnorm{\boldsymbol{u}-\boldsymbol{u}_h}_v
    \le 2(1+\tilde{c}_{f}\tilde{c}_{ae}^{-1}) \tnorm{\boldsymbol{e}_u^I}_{v'},
  \end{equation}
  so that \cref{eq:apriori-u} follows by using
  \cref{eq:interpestimate_u}.

  We proceed with proving \cref{eq:apriori-p}. Set
  $\boldsymbol{q}_h = 0$ in \cref{eq:errorequation}. Then, by
  \cref{lem:boundedness_th,lem:boundedness_ahwuv},
  \begin{equation*}
    \begin{split}
      b_h(\boldsymbol{v}_h, \boldsymbol{e}_p^h)      
      =&
      a_h^L(\boldsymbol{e}_u^I, \boldsymbol{v}_h)
      - a_h^L(\boldsymbol{e}_u^h, \boldsymbol{v}_h)
      - t_h(u;\boldsymbol{e}_u^h, \boldsymbol{v}_h)
      \\
      &+ t_h(u;\boldsymbol{e}_u^I, \boldsymbol{v}_h)    
      + t_h(u;\boldsymbol{u}_h, \boldsymbol{v}_h) 
      - t_h(u_h;\boldsymbol{u}_h, \boldsymbol{v}_h)
      \\
      =&
      a_h(u;\boldsymbol{e}_u^I, \boldsymbol{v}_h)
      - a_h(u;\boldsymbol{e}_u^h, \boldsymbol{v}_h)
      + t_h(u;\boldsymbol{u}_h, \boldsymbol{v}_h) 
      - t_h(u_h;\boldsymbol{u}_h, \boldsymbol{v}_h)
      \\
      \le & \tilde{c}_{f} \mu \del[0]{\tnorm{\boldsymbol{e}_u^I}_{v'} + \tnorm{\boldsymbol{e}_u^h}_{v}}\tnorm{\boldsymbol{v}_h}_{v}
      + c_w \norm[0]{u-u_h}_{1,h,\Omega^s}\tnorm{\boldsymbol{u}_h}_{v}\tnorm{\boldsymbol{v}_h}_{v}.
    \end{split}
  \end{equation*}
  By \cref{eq:cwuhupbound,eq:dpoincareineq},
  \begin{equation*}
    \begin{split}
      b_h(\boldsymbol{v}_h, \boldsymbol{e}_p^h)      
      \le & \tilde{c}_f \mu \del[0]{\tnorm{\boldsymbol{e}_u^I}_{v'} + \tnorm{\boldsymbol{e}_u^h}_{v}}\tnorm{\boldsymbol{v}_h}_{v}
      + \tfrac{1}{2} \tilde{c}_{ae}\mu \norm[0]{u-u_h}_{1,h,\Omega^s}\tnorm{\boldsymbol{v}_h}_{v}
      \\
      \le & \tilde{c}_f \mu \del[0]{\tnorm{\boldsymbol{e}_u^I}_{v'} + \tnorm{\boldsymbol{e}_u^h}_{v}}\tnorm{\boldsymbol{v}_h}_{v}
      + \tfrac{1}{2} \tilde{c}_{ae}\mu \del[0]{\tnorm{\boldsymbol{e}_u^I}_{v,s} + \tnorm{\boldsymbol{e}_u^h}_{v,s}} \tnorm{\boldsymbol{v}_h}_{v}
      \\
      \le & \del[0]{\tilde{c}_f + \tfrac{1}{2} \tilde{c}_{ae}}\mu \del[0]{\tnorm{\boldsymbol{e}_u^I}_{v'} + \tnorm{\boldsymbol{e}_u^h}_{v}}\tnorm{\boldsymbol{v}_h}_{v}.
    \end{split}
  \end{equation*}
  By the inf-sup condition in \cref{lem:infsupbh},
  \begin{equation}
    \label{eq:ephpbound}
     \tnorm{\boldsymbol{e}_p^h}_p \le c_{bb}^{-1}
    \sup_{\substack{\boldsymbol{v}_h \in \boldsymbol{X}_h \\ \boldsymbol{v}_h \ne 0}}
    \frac{b_h(\boldsymbol{v}_h, \boldsymbol{e}_p^h)}{\tnorm{\boldsymbol{v}_h}_{v}}
    \le \del[0]{\tilde{c}_f + \tfrac{1}{2} \tilde{c}_{ae}}c_{bb}^{-1}\mu \del[0]{\tnorm{\boldsymbol{e}_u^I}_{v'} + \tnorm{\boldsymbol{e}_u^h}_{v}}.
  \end{equation}
  Applying the triangle inequality to
  $\tnorm{\boldsymbol{p}-\boldsymbol{p}_h}_p$ and combining
  \cref{eq:ephpbound} with \cref{eq:approxininterpeu} we find 
  \begin{equation}
    \label{eq:pressureapproxerrorterm}
    \tnorm{\boldsymbol{p}-\boldsymbol{p}_h}_p
    \le \tnorm{\boldsymbol{e}_p^I}_p + \del[0]{2\tilde{c}_f + \tilde{c}_{ae}}c_{bb}^{-1}
    \del[0]{1 + \tilde{c}_f\tilde{c}_{ae}^{-1}} \mu \tnorm{\boldsymbol{e}_u^I}_{v'},
  \end{equation}
  so that \cref{eq:apriori-p} follows using
  \cref{eq:interpestimate_p,eq:interpestimate_u}. \qed
\end{proof}

\begin{remark}
  \label{rem:parameters}
  The velocity error bound \cref{eq:apriori-u} is independent of the
  pressure and independent of the inverse of the viscosity; the
  discretization is pressure-robust. However, note that
  $\tilde{c}_f\tilde{c}_{ae}^{-1} =
  \mathcal{O}(\kappa_{\max}/\kappa_{\min})$ for small $\kappa_{\min}$
  and large $\kappa_{\max}$. Therefore, for small $\kappa_{\min}$ and
  large $\kappa_{\max}$, by \cref{eq:uuheuI}, the constant in the
  velocity approximation error \cref{eq:apriori-u} increases linearly
  with $\kappa_{\max}/\kappa_{\min}$. The dependence of the pressure
  error on $\kappa_{\max}/\kappa_{\min}$ is small for small enough
  $\mu$, see \cref{eq:pressureapproxerrorterm}.
\end{remark}

\section{Numerical examples}
\label{sec:numerical_examples}

We now present numerical examples in which solutions to the
Navier--Stokes/Darcy problem \cref{eq:system,eq:interface} are
approximated by solutions to the HDG discretization
\cref{eq:hdg}. The HDG method is implemented using the finite
element software Netgen/NGSolve \cite{Schoberl:1997,Schoberl:2014}.

\subsection{Example 1: Manufactured Solution}
\label{ss:ratesconvergence}

We consider here a manufactured solution on the domain
$\overline{\Omega} = [0, 1] \times [-1, 1]$ such that
$\overline{\Omega}^s = [0, 1] \times [0, 1]$ and
$\overline{\Omega}^d = [0,1] \times [-1, 0]$. We consider two values
for $\kappa$, namely, $\kappa = \kappa_1(x_1)\mathbb{I}$ with
$\kappa_1 = \alpha^2(\pi x_1 + 1)^2/4$ and
$\kappa = \kappa_2(x)\mathbb{I}$ with
$\kappa_2 = \kappa_1(x_1)\exp(-15 \sin^2(10x_2) )$. We furthermore consider
the manufactured solution
\begin{subequations}
  \begin{align*}
    u^s
    &=
      \begin{bmatrix}
        \pi x_1\cos(\pi x_1x_2) + 1
        \\
        -\pi x_2\cos(\pi x_1x_2) + 2x_1
      \end{bmatrix},
    &&
    \\
    p^s &= \mu(1-\pi)\cos(\pi x_1 x_2) + \sin(\tfrac{1}{2}\pi x_2)/\mu,
    &
    p^d &=-\frac{8\mu x_1 x_2}{(\pi x_1 + 1)^2\alpha^2} + \mu\cos(\pi x_1x_2).
  \end{align*}
\end{subequations}
This manufactured solution is used to set
$u^d=-\kappa\mu^{-1}\nabla p^d$, the source terms, $f^s$ and $f^d$,
and inhomogeneous boundary conditions. In our simulations we set
$\alpha = 1$ and consider $\mu=10^{-1}$ and $\mu=10^{-3}$. In our
discrete function spaces we consider $k=1$ (corresponding to the
lowest order polynomial approximation in which the cell pressure is
approximated by piecewise constants and all other unknowns by
piecewise linear polynomial approximations), and higher-order accurate
approximations with $k=2$, and $k=3$. We choose the penalty parameter
in \cref{eq:ah_s} as $\beta = 8k^2$.

Define
$\norm{v}_E^2 := \del[1]{\sum_{K \in \mathcal{T}^s} |v|_{1,K}^2 +
  \norm[0]{v}_{\Omega^d}^2}$. Using $\kappa_1$ we observe in
\cref{fig:errorsuEpL2k123EuL2p} that the velocity in the
$\norm{\cdot}_E$-norm and pressure both converge at rate $k$, as
expected from \cref{thm:energyestimate}. When $k=2$ and $k=3$ the
errors in the velocity are significantly larger using $\kappa_2$
compared to $\kappa_1$. This is again as expected from
\cref{thm:energyestimate} since for $\kappa_2$ we have
$\kappa_{\max}/\kappa_{\min} \approx 5.6 \cdot 10^7$ which is
significantly larger than
$\kappa_{\max}/\kappa_{\min} \approx 1.7 \cdot 10^1$ when using
$\kappa_1$ (see also \cref{rem:parameters}). Interestingly, the
velocity and pressure errors when $k=1$ do not seem to depend on
$\kappa_{\max}/\kappa_{\min}$. Note furthermore that when reducing the
viscosity by a factor of 100 from $\mu=10^{-1}$ to $\mu=10^{-3}$, the
error in the pressure increases approximately by a factor of 100, but
the error in the velocity is unaffected by changing viscosity. This is
also as expected from \cref{thm:energyestimate}, i.e., our
discretization is pressure-robust.

In \cref{rem:parameters} we pointed out that for small enough
viscosity an increase in $\kappa_{\max}/\kappa_{\min}$ only has a
small effect on the pressure error. In the right column of plots in
\cref{fig:errorsuEpL2k123EuL2p} we indeed observe that the effect of
$\kappa_{\max}/\kappa_{\min}$ is negligible for $\mu=10^{-3}$, but
less so for $\mu=10^{-1}$ in the pre-asymptotic regime.

Finally, in \cref{fig:errorsuL2} we plot the velocity error in the
$L^2$-norm. We observe optimal $k+1$ rates of convergence when $k=2$
and $k=3$. For $k=1$, $\mu=10^{-3}$, and small
$\kappa_{\max}/\kappa_{\min}$ ratio, we observe a rate of convergence
between 1.6 and 1.9. The velocity error magnitude in the $L^2$-norm is
independent of the viscosity, but clearly increases with increasing
$\kappa_{\max}/\kappa_{\min}$ ratio.

\begin{figure}
  \centering
  \subfloat[Velocity error $\norm{u-u_h}_E$, $k=1$. \label{fig:uuhEk1}]{\includegraphics[width=0.48\textwidth]{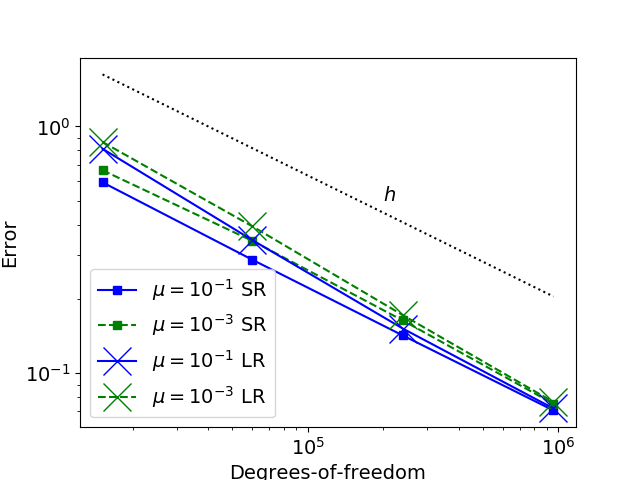}}
  \quad
  \subfloat[Pressure error $\norm{p-p_h}_{\Omega}$, $k=1$. \label{fig:pphl2k1}]{\includegraphics[width=0.48\textwidth]{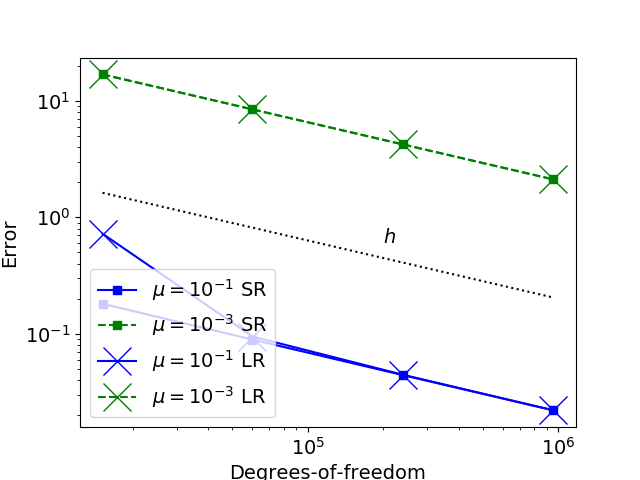}}
  \\
  \subfloat[Velocity error $\norm{u-u_h}_E$, $k=2$. \label{fig:uuhEk2}]{\includegraphics[width=0.48\textwidth]{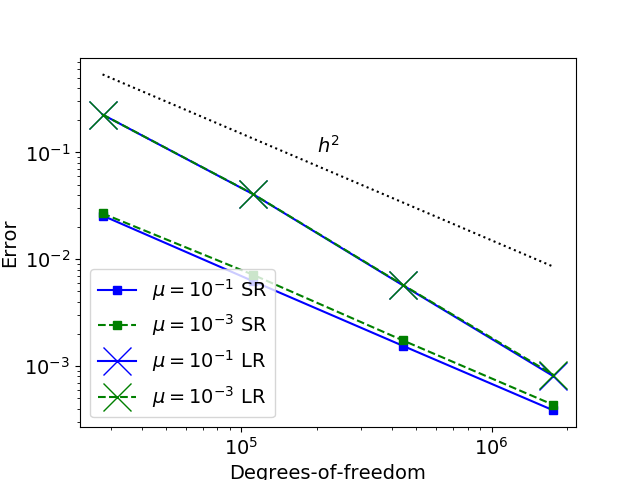}}
  \quad
  \subfloat[Pressure error $\norm{p-p_h}_{\Omega}$, $k=2$. \label{fig:pphl2k2}]{\includegraphics[width=0.48\textwidth]{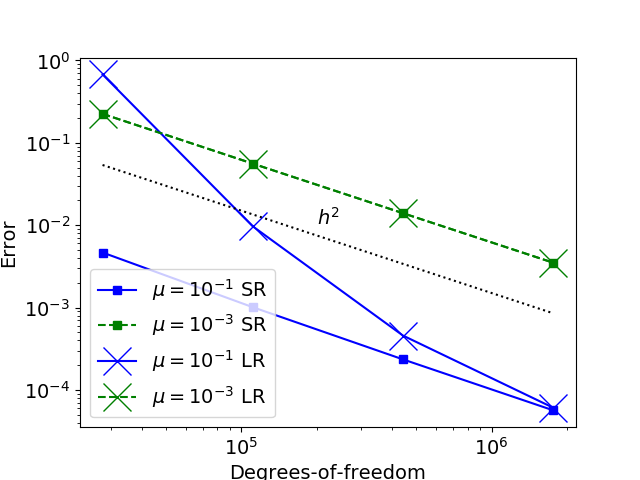}}
  \\
  \subfloat[Velocity error $\norm{u-u_h}_E$, $k=3$. \label{fig:uuhEk3}]{\includegraphics[width=0.48\textwidth]{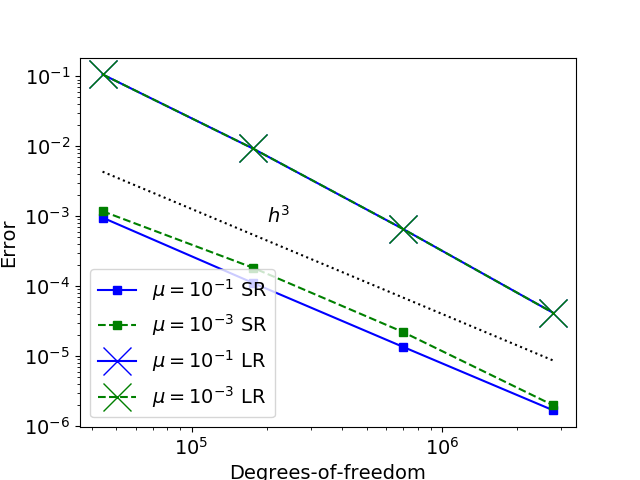}}
  \quad
  \subfloat[Pressure error $\norm{p-p_h}_{\Omega}$, $k=3$. \label{fig:pphl2k3}]{\includegraphics[width=0.48\textwidth]{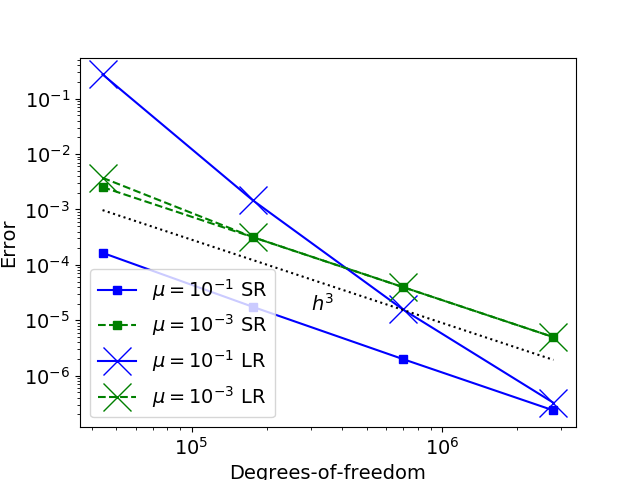}}
  \caption{The velocity error in the $\norm{\cdot}_E$- and the
    pressure error in the $L^2$-norm for the test case of
    \cref{ss:ratesconvergence}. Here the blue lines correspond to
    $\mu = 10^{-1}$ and the green dashed lines correspond to
    $\mu = 10^{-3}$. The square symbols correspond to
    $\kappa = \kappa_1\mathbb{I}$ with a small
    $\kappa_{\max}/\kappa_{\min}$ ratio (SR) while the $\times$ symbols
    correspond to $\kappa = \kappa_2\mathbb{I}$ with a large
    $\kappa_{\max}/\kappa_{\min}$ ratio (LR).}
  \label{fig:errorsuEpL2k123EuL2p}
\end{figure}

\begin{figure}
  \centering
  \subfloat[Velocity error $\norm{u-u_h}_{\Omega}$, $k=1$. \label{fig:uuhl2k1o}]{\includegraphics[width=0.48\textwidth]{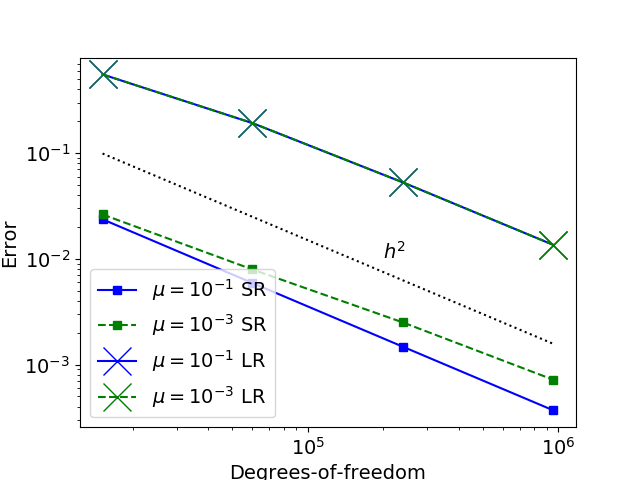}}
  \quad
  \subfloat[Velocity error $\norm{u-u_h}_{\Omega}$, $k=2$. \label{fig:uuhl2k2o}]{\includegraphics[width=0.48\textwidth]{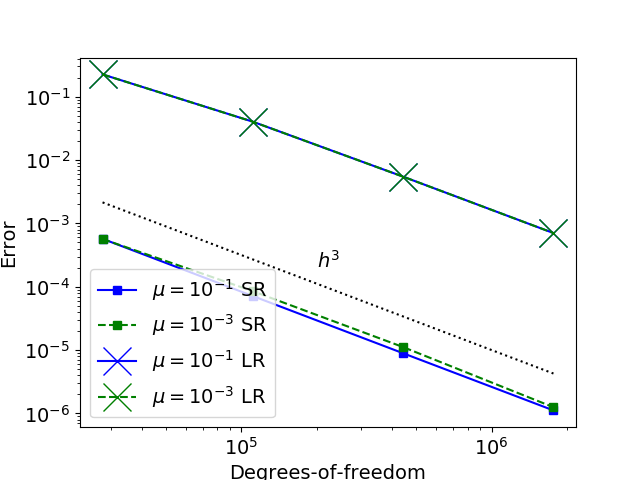}}
  \\
  \subfloat[Velocity error $\norm{u-u_h}_{\Omega}$, $k=3$. \label{fig:uuhl2k3o}]{\includegraphics[width=0.48\textwidth]{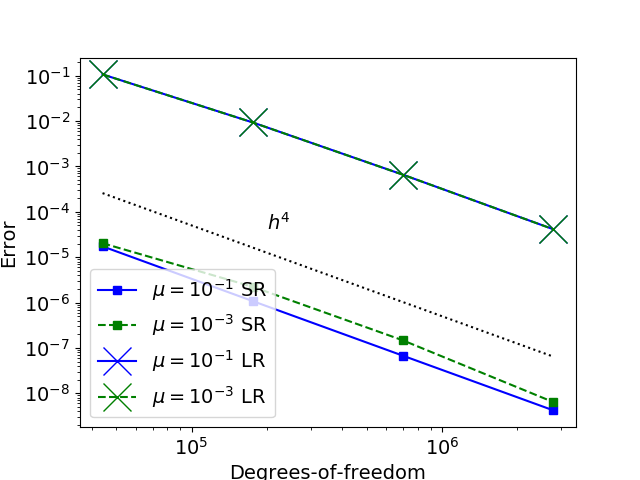}}
  \caption{The velocity error in the $L^2$-norm for the test case of
    \cref{ss:ratesconvergence}. Here the blue lines correspond to
    $\mu = 10^{-1}$ and the green dashed lines correspond to
    $\mu = 10^{-3}$. The square symbols correspond to
    $\kappa = \kappa_1\mathbb{I}$ with a small
    $\kappa_{\max}/\kappa_{\min}$ ratio (SR) while the $\times$
    symbols correspond to $\kappa = \kappa_2\mathbb{I}$ with a large
    $\kappa_{\max}/\kappa_{\min}$ ratio (LR).}
  \label{fig:errorsuL2}
\end{figure}

\subsection{Example 2: Coupled surface/subsurface flow with randomly generated permeability field}
\label{ss:surf_subsurf_r}

We consider now a Navier--Stokes/Darcy problem similar to a problem
proposed in \cite[Section 8.2]{Girault:2009}. We consider the domain
$\overline{\Omega} = [0, 1] \times [0, 1]$ such that
$\overline{\Omega}^s = [0, 1] \times [0.6, 1]$ and
$\overline{\Omega}^d = [0,1] \times [0, 0.6]$ and impose the following
boundary conditions:
\begin{align*}
  u &= (\sin( (\pi/8)(10 x_2-6))(1-x_1/5), 0) && \text{on } \Gamma^s,
  \\
  u \cdot n &= 0 && \text{on } \cbr[0]{x \in \Gamma^d:\, x_1 = 0 \text{ or } x_1 = 1},
  \\
  p &= 2-x_1 && \text{on } \cbr[0]{x \in \Gamma^d:\, x_2 = 0}.
\end{align*}
We take $\alpha = 1$, $f^s=0$, $f^d=0$, consider the solution for
$\mu=1$ and $\mu=10^{-2}$, and set the permeability on each element of
the mesh in $\Omega^d$ to a constant such that
$\mu^{-1}\kappa = 10^{-r}$ with $r$ a random number in the interval
$[2, 6]$ (see \cref{fig:permeability}). We furthermore set $k=2$,
$\beta=8k^2$, and compute our solution on a mesh consisting of 92,672
triangles (corresponding to a total of 2,143,476
degrees-of-freedom).

\begin{figure}
  \centering
  \includegraphics[width=0.48\textwidth]{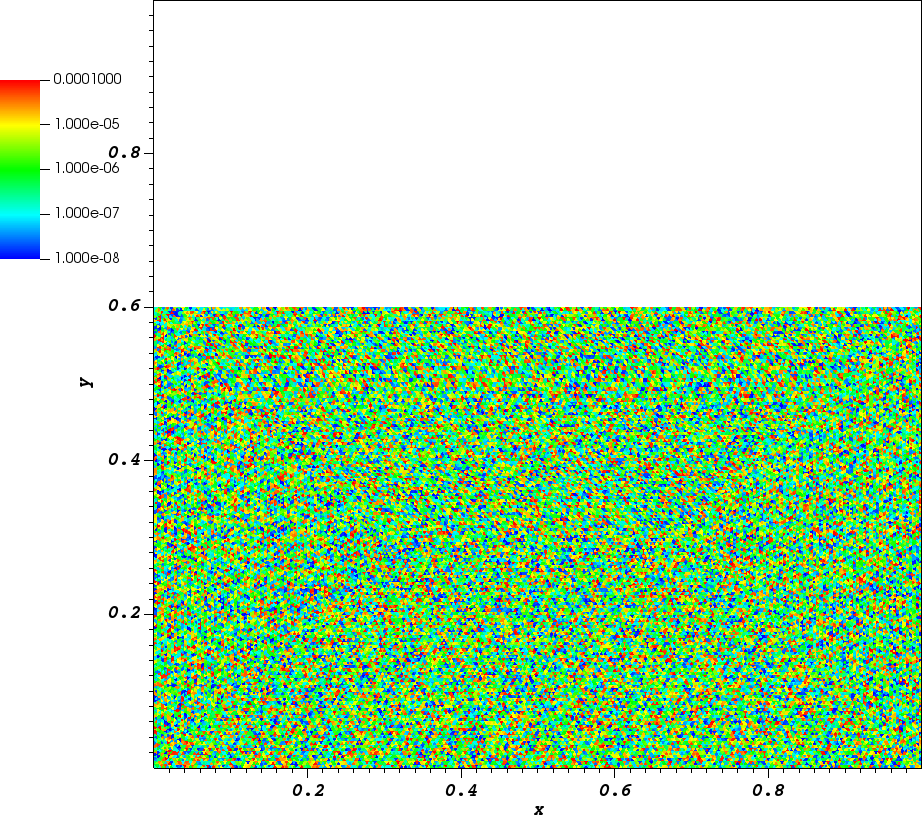}
  \caption{The random permeability when $\mu = 0.01$ for the test case
    described in \cref{ss:surf_subsurf_r}.}
  \label{fig:permeability}
\end{figure}

In \cref{fig:surf_subsurf_r} we plot the magnitude and streamlines of
the velocity and pressure fields computed using $\mu=1$ and
$\mu=0.01$. We observe, for both values of viscosity, that away from
the interface $\Gamma^I$ the fluid flows freely in $\Omega^s$. Fluid
in $\Omega^s$ close to the interface percolates through into the
subsurface region $\Omega^d$. The flow patterns observed in
\cref{fig:surf_subsurf_r} are similar to those observed in
\cite[Section 8.3]{Girault:2009}. Let us finally remark that for
$\mu=10^{-2}$, $\kappa \in [10^{-8},10^{-4}]$. Like the DG method
proposed in \cite{Girault:2009}, our HDG method is able to handle
highly discontinuous permeability.

\begin{figure}
  \centering
  \subfloat[Velocity ($\mu=1$). \label{fig:tc_velr_mu1}]{\includegraphics[width=0.48\textwidth]{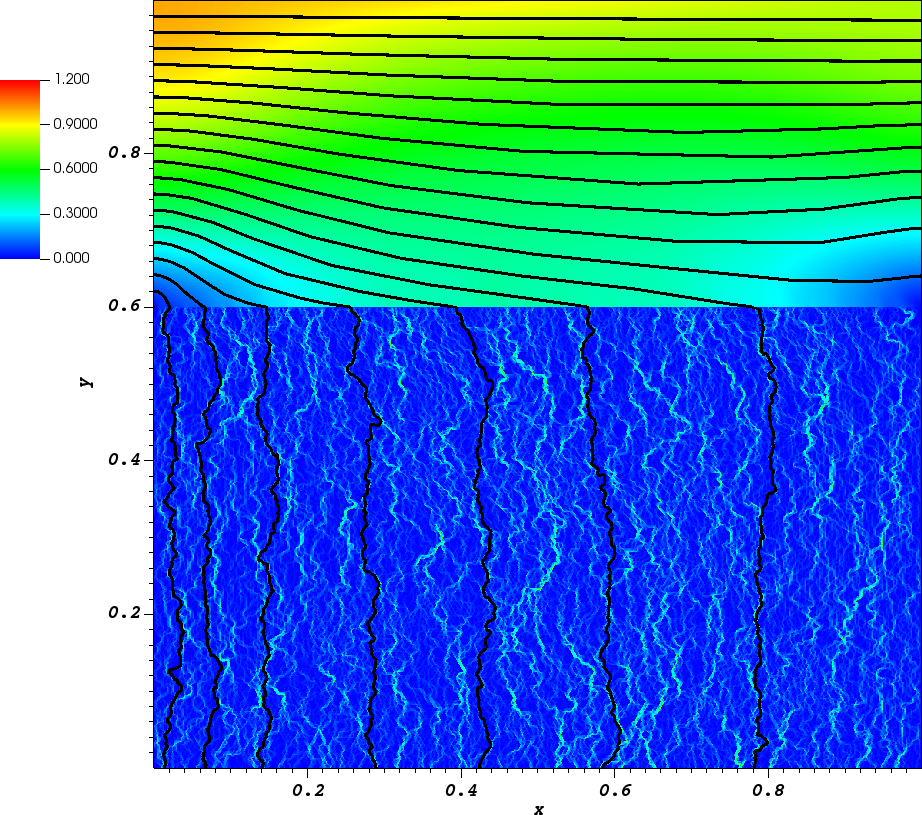}}
  \quad
  \subfloat[Velocity ($\mu=0.01$). \label{fig:tc_velr_mu001}]{\includegraphics[width=0.48\textwidth]{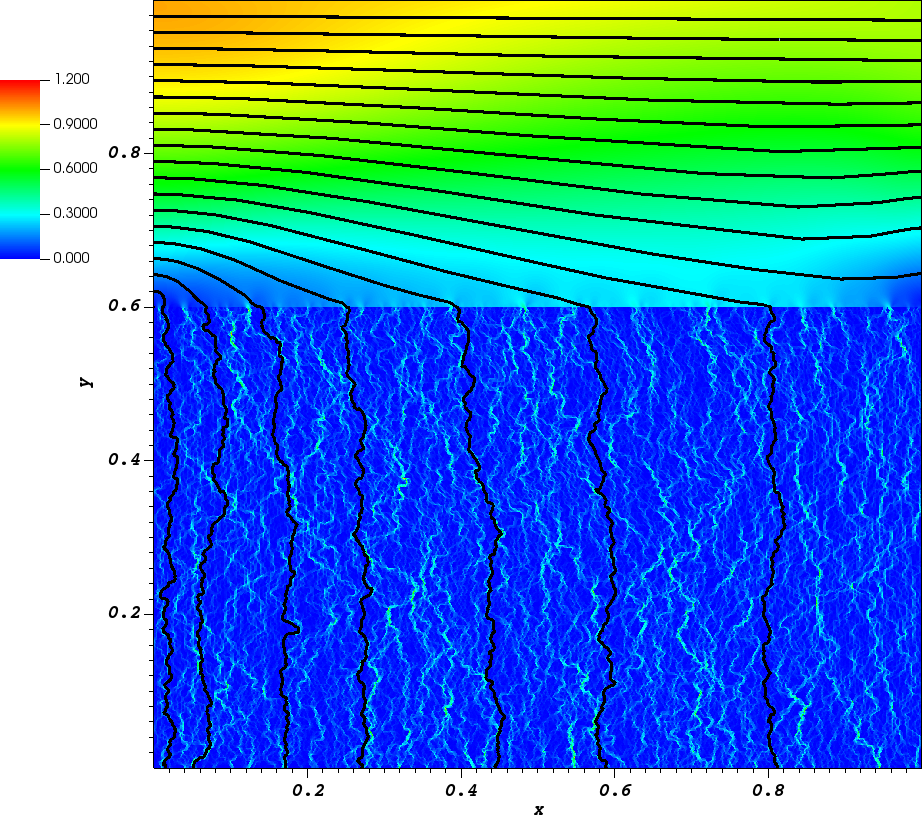}}
  \\
  \subfloat[Pressure ($\mu=1$). \label{fig:tc_presr_mu1}]{\includegraphics[width=0.48\textwidth]{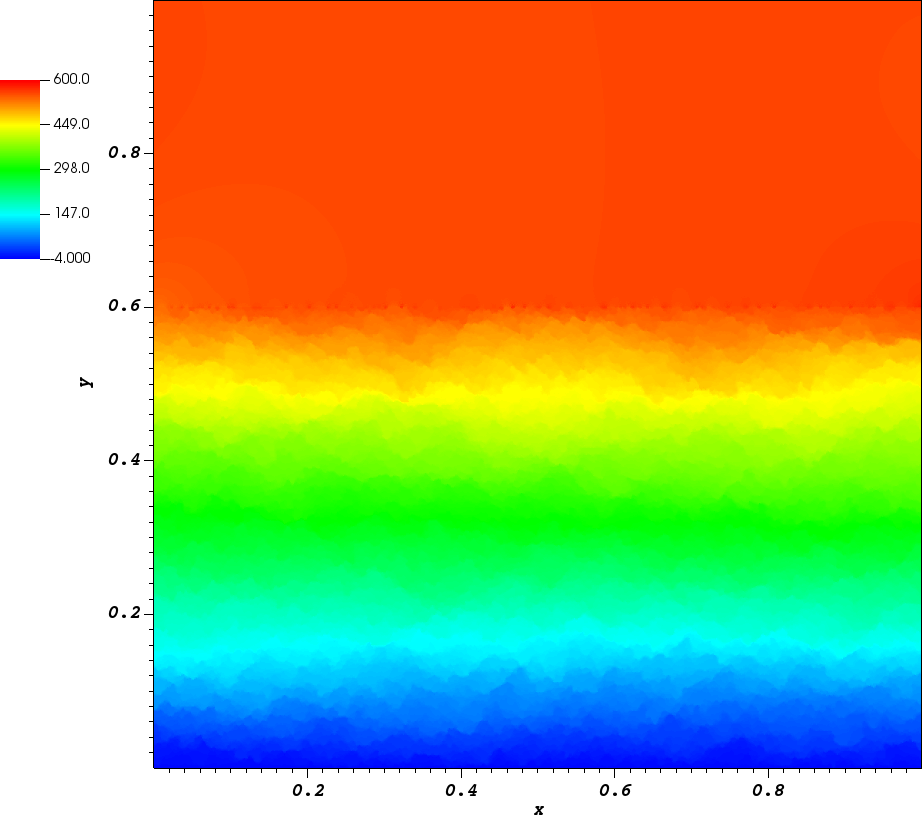}}
  \quad
  \subfloat[Pressure ($\mu=0.01$). \label{fig:tc_presr_mu001}]{\includegraphics[width=0.48\textwidth]{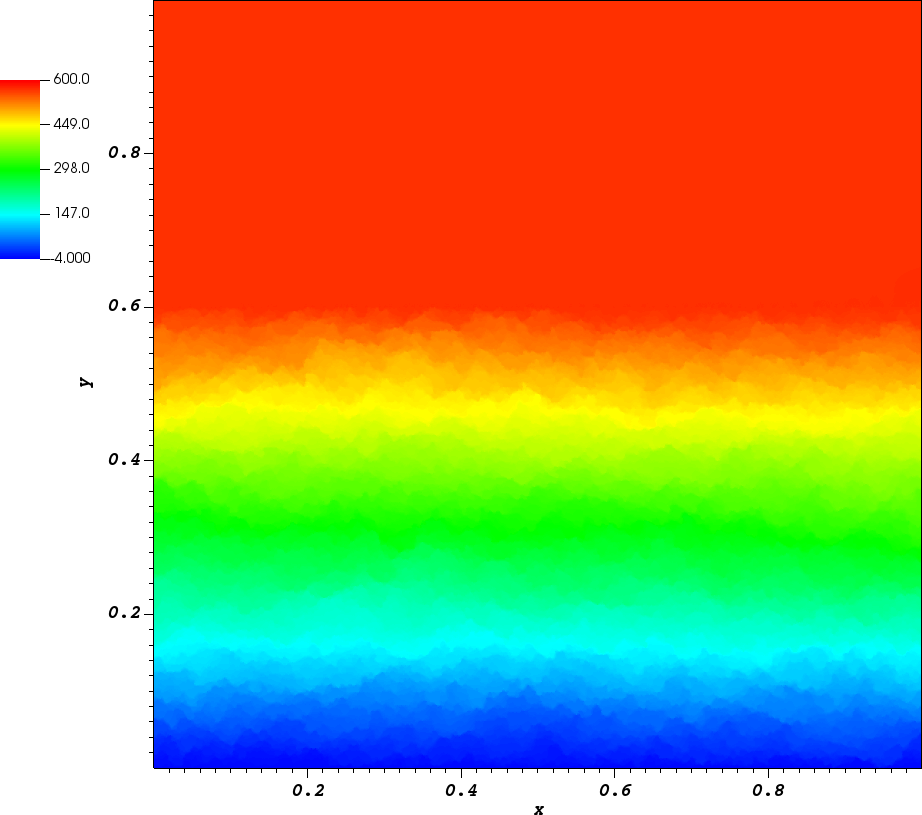}}
  \caption{The magnitude and streamlines of the velocity (top row) and
    pressure magnitude (bottom row) when $\mu=1$ and $\mu=0.01$ for
    the test case described in~\cref{ss:surf_subsurf_r}. We remark
    that the streamlines should be followed from left to right.}
  \label{fig:surf_subsurf_r}
\end{figure}

\section{Conclusions}
\label{ss:conclusion}

We have introduced and analyzed a strongly conservative HDG method for
the Navier--Stokes equations coupled to the Darcy equations by the
Beavers--Joseph--Saffman interface condition. The discretization
results in a velocity field that is globally divergence-conforming and
pointwise divergence-free in the Navier--Stokes region. This allows
for a locally momentum conserving discretization of the Navier--Stokes
equation. (If the divergence-free constraint is satisfied only weakly,
local momentum conservation needs to be sacrificed for the
discretization to be stable \cite{Cockburn:2004b}.) A further property
of the discretization is that the mass equation in the Darcy region is
satisfied pointwise if the source/sink term lies in the discrete
pressure space.

Optimal rates of convergence were proven for the velocity and
pressure. Additionally, the velocity error is independent of the
pressure and viscosity, i.e., the coupled discretization is
pressure-robust.

\subsubsection*{Acknowledgements}

SR gratefully acknowledges support from the Natural Sciences and
Engineering Research Council of Canada through the Discovery Grant
program (RGPIN-05606-2015).

\bibliographystyle{abbrvnat}
\bibliography{references}
\end{document}